\documentclass[10pt]{article}
\usepackage{amsmath}
\usepackage{amssymb}
\usepackage{amsthm}
\usepackage{mathrsfs}
\numberwithin{equation}{section}

\begin{document}

\title{Two-weight, weak type norm inequalities for a class of sublinear operators on weighted Morrey and amalgam spaces}
\author{Hua Wang \footnote{E-mail address: wanghua@pku.edu.cn.}\\
\footnotesize{College of Mathematics and Econometrics, Hunan University, Changsha 410082, P. R. China}\\
\footnotesize{\&~Department of Mathematics and Statistics, Memorial University, St. John's, NL A1C 5S7, Canada}}
\date{}
\maketitle

\begin{abstract}
Let $\mathcal T_\alpha~(0\leq\alpha<n)$ be a class of sublinear operators satisfying certain size conditions introduced by Soria and Weiss, and let $[b,\mathcal T_\alpha]~(0\leq\alpha<n)$ be the commutators generated by $\mathrm{BMO}(\mathbb R^n)$ functions and $\mathcal T_\alpha$. This paper is concerned with two-weight, weak type norm estimates for these sublinear operators and their commutators on the weighted Morrey and amalgam spaces. Some boundedness criterions for such operators are given, under the assumptions that weak-type norm inequalities on weighted Lebesgue spaces are satisfied. As applications of our main results, we can obtain the weak-type norm inequalities for several integral operators as well as the corresponding commutators in the framework of weighted Morrey and amalgam spaces.
\end{abstract}

\section{Introduction and main results}

\subsection{Sublinear operators}

Let $\mathbb R^n$ be the $n$-dimensional Euclidean space equipped with the Euclidean norm $|\cdot|$ and the Lebesgue measure $dx$. Suppose that $\mathcal T$ represents a linear or a sublinear operator, which satisfies that for any $f\in L^1(\mathbb R^n)$ with compact support and $x\notin\mathrm{supp}\, f$,
\begin{equation}\label{sublinear}
\big|\mathcal Tf(x)\big|\leq c_1\int_{\mathbb R^n}\frac{|f(y)|}{|x-y|^n}dy,
\end{equation}
where $c_1$ is a universal constant independent of $f$ and $x\in\mathbb R^n$. The size condition \eqref{sublinear} was first introduced by Soria and Weiss in \cite{soria}. It can be proved that \eqref{sublinear} is satisfied by many integral operators in harmonic analysis, such as the Hardy--Littlewood maximal operator, Calder\'on--Zygmund singular integral operators, Ricci--Stein's oscillatory singular integrals and Bochner--Riesz operators at the critical index and so on.

Similarly, for given $0<\gamma<n$, we assume that $\mathcal T_\gamma$ represents a linear or a sublinear operator with order $\gamma$, which satisfies that for any $f\in L^1(\mathbb R^n)$ with compact support and $x\notin\mathrm{supp}\, f$,
\begin{equation}\label{frac sublinear}
\big|\mathcal T_\gamma f(x)\big|\leq c_2\int_{\mathbb R^n}\frac{|f(y)|}{|x-y|^{n-\gamma}}dy,
\end{equation}
where $c_2$ is also a universal constant independent of $f$ and $x\in\mathbb R^n$. It can be easily checked that (\ref{frac sublinear}) is satisfied by some important operators such as the fractional maximal operator, Riesz potential operators and fractional oscillatory singular integrals and so on.

Let $b$ be a locally integrable function on $\mathbb R^n$, suppose that the commutator operator $[b,\mathcal T]$ stands for a linear or a sublinear operator, which satisfies that for any $f\in L^1(\mathbb R^n)$ with compact support and $x\notin\mathrm{supp}\, f$,
\begin{equation}\label{sublinear commutator}
\big|[b,\mathcal T](f)(x)\big|\leq c_3\int_{\mathbb R^n}\frac{|b(x)-b(y)|\cdot|f(y)|}{|x-y|^n}dy,
\end{equation}
where $c_3$ is an absolute constant independent of $f$ and $x\in\mathbb R^n$.

Similarly, for given $0<\gamma<n$, we assume that the commutator operator $[b,\mathcal T_\gamma]$ stands for a linear or a sublinear operator, which satisfies that for any $f\in L^1(\mathbb R^n)$ with compact support and $x\notin\mathrm{supp}\, f$,
\begin{equation}\label{frac sublinear commutator}
\big|[b,\mathcal T_\gamma](f)(x)\big|\leq c_4\int_{\mathbb R^n}\frac{|b(x)-b(y)|\cdot|f(y)|}{|x-y|^{n-\gamma}}dy,
\end{equation}
where $c_4$ is also an absolute constant independent of $f$ and $x\in\mathbb R^n$. Clearly, based on the above assumptions,
\begin{equation*}
\mathcal T_0=\mathcal T \qquad \& \qquad [b,\mathcal T_0]=[b,\mathcal T].
\end{equation*}
\subsection{Weighted Morrey spaces}
\newtheorem{theorem}{Theorem}[section]

\newtheorem{lemma}{Lemma}[section]

\newtheorem{corollary}[theorem]{Corollary}

\newtheorem{defn}{Definition}[section]
The classical Morrey space was introduced by Morrey \cite{morrey} in connection with elliptic partial differential equations. Let $1\leq p<\infty$ and $0\leq\lambda\leq n$. We recall that a real-valued function $f$ is said to belong to the space $L^{p,\lambda}$ on the $n$-dimensional Euclidean space $\mathbb R^n$, if the following norm is finite:
\begin{equation*}
\|f\|_{L^{p,\lambda}}:=\sup_{(x,r)\in\mathbb R^n\times(0,\infty)}\bigg(r^{\lambda-n}\int_{B(x,r)}|f(y)|^p\,dy\bigg)^{1/p},
\end{equation*}
where $B(x,r)=\big\{y\in\mathbb R^n:|x-y|<r\big\}$ is the Euclidean ball with center $x\in\mathbb R^n$ and radius $r\in(0,\infty)$. In particular, one has
\begin{equation*}
L^{p,0}=L^\infty \qquad\&\qquad L^{p,n}=L^p.
\end{equation*}

In \cite{komori}, Komori and Shirai considered the weighted case, and gave the definitions of the weighted Morrey spaces as follows.
\begin{defn}
Let $1<p<\infty$ and $0\leq\kappa<1$. For two weights $w$ and $\nu$ on $\mathbb R^n$, the weighted Morrey space $L^{p,\kappa}(\nu,w)$ is defined by
\begin{equation*}
L^{p,\kappa}(\nu,w):=\Big\{f\in L^p_{loc}(\nu):\big\|f\big\|_{L^{p,\kappa}(\nu,w)}<\infty\Big\},
\end{equation*}
where the norm is given by
\begin{equation*}
\big\|f\big\|_{L^{p,\kappa}(\nu,w)}
:=\sup_{Q\subset\mathbb R^n}\bigg(\frac{1}{w(Q)^{\kappa}}\int_Q |f(x)|^p\nu(x)\,dx\bigg)^{1/p},
\end{equation*}
and the supremum is taken over all cubes $Q$ in $\mathbb R^n$.
\end{defn}

\begin{defn}
Let $1<p<\infty$, $0\leq\kappa<1$ and $w$ be a weight on $\mathbb R^n$. We define the weighted weak Morrey space $WL^{p,\kappa}(w)$ as the set of all measurable functions $f$ satisfying
\begin{equation*}
\big\|f\big\|_{WL^{p,\kappa}(w)}:=\sup_{Q\subset\mathbb R^n}\sup_{\sigma>0}\frac{1}{w(Q)^{{\kappa}/p}}\sigma
\cdot \Big[w\big(\big\{x\in Q:|f(x)|>\sigma\big\}\big)\Big]^{1/p}<\infty.
\end{equation*}
\end{defn}
By definition, it is clear that
\begin{equation*}
L^{p,0}(\nu,w)=L^p(\nu)\qquad \& \qquad WL^{p,0}(w)=WL^p(w).
\end{equation*}

\subsection{Weighted amalgam spaces}

Let $1\leq p,q\leq\infty$, a function $f\in L^p_{loc}(\mathbb R^n)$ is said to be in the Wiener amalgam space $(L^p,L^q)(\mathbb R^n)$ of $L^p(\mathbb R^n)$ and $L^q(\mathbb R^n)$, if the function $y\mapsto\|f(\cdot)\cdot\chi_{B(y,1)}\|_{L^p}$ belongs to $L^q(\mathbb R^n)$, where $B(y,1)$ is an open ball in $\mathbb R^n$ centered at $y$ with radius $1$, $\chi_{B(y,1)}$ is the characteristic function of the ball $B(y,1)$, and $\|\cdot\|_{L^p}$ is the usual Lebesgue norm in $L^p(\mathbb R^n)$. In \cite{fofana}, Fofana introduced a new class of function spaces $(L^p,L^q)^{\alpha}(\mathbb R^n)$ which turn out to be the subspaces of $(L^p,L^q)(\mathbb R^n)$. More precisely, for $1\leq p,q,\alpha\leq\infty$, we define the amalgam space $(L^p,L^q)^{\alpha}(\mathbb R^n)$ of $L^p(\mathbb R^n)$ and $L^q(\mathbb R^n)$ as the set of all measurable functions $f$ satisfying $f\in L^p_{loc}(\mathbb R^n)$ and $\big\|f\big\|_{(L^p,L^q)^{\alpha}(\mathbb R^n)}<\infty$, where
\begin{equation*}
\begin{split}
\big\|f\big\|_{(L^p,L^q)^{\alpha}(\mathbb R^n)}
:=&\sup_{r>0}\left\{\int_{\mathbb R^n}\Big[\big|B(y,r)\big|^{1/{\alpha}-1/p-1/q}\big\|f\cdot\chi_{B(y,r)}\big\|_{L^p(\mathbb R^n)}\Big]^qdy\right\}^{1/q}\\
=&\sup_{r>0}\Big\|\big|B(y,r)\big|^{1/{\alpha}-1/p-1/q}\big\|f\cdot\chi_{B(y,r)}\big\|_{L^p(\mathbb R^n)}\Big\|_{L^q(\mathbb R^n)},
\end{split}
\end{equation*}
with the usual modification when $p=\infty$ or $q=\infty$, and $|B(y,r)|$ is the Lebesgue measure of the ball $B(y,r)$. As it was shown in \cite{fofana} that the space $(L^p,L^q)^{\alpha}(\mathbb R^n)$ is non-trivial if and only if $p\leq\alpha\leq q$; thus in the remaining of this paper we will always assume that this condition $p\leq\alpha\leq q$ is satisfied. Let us consider the following special cases:
\begin{enumerate}
  \item If we take $p=q$, then $p=\alpha=q$. It is easy to check that
\begin{equation*}
\begin{split}
&\Big\|\big|B(y,r)\big|^{-1/p}\big\|f\cdot\chi_{B(y,r)}\big\|_{L^p(\mathbb R^n)}\Big\|_{L^p(\mathbb R^n)}\\
=&\bigg[\int_{\mathbb R^n}\big|B(y,r)\big|^{-1}\bigg(\int_{\mathbb R^n}|f(x)|^p\cdot\chi_{B(y,r)}\,dx\bigg)\,dy\bigg]^{1/p}\\
=&\bigg[\int_{\mathbb R^n}\big|B(y,r)\big|^{-1}\bigg(\int_{B(x,r)}|f(x)|^p\,dy\bigg)\,dx\bigg]^{1/p}\\
=&\bigg[\int_{\mathbb R^n}|f(x)|^p\,dx\bigg]^{1/p}.
\end{split}
\end{equation*}
Hence, the amalgam space $(L^p,L^q)^{\alpha}(\mathbb R^n)$ is equal to the Lebesgue space $L^p(\mathbb R^n)$ with the same norms provided that $p=\alpha=q$.
  \item If $q=\infty$, then we can see that the amalgam space $(L^p,L^q)^{\alpha}(\mathbb R^n)$ is equal to the Morrey space $L^{p,\lambda}(\mathbb R^n)$ with equivalent norms, where $\lambda={(pn)}/{\alpha}$.
\end{enumerate}

In this paper, we will consider the weighted version of $(L^p,L^q)^{\alpha}(\mathbb R^n)$.
\begin{defn}\label{amalgam}
Let $1\leq p\leq\alpha\leq q\leq\infty$, and let $\nu,w,\mu$ be three weights on $\mathbb R^n$. We denote by $(L^p,L^q)^{\alpha}(\nu,w;\mu)$ the weighted amalgam space, the space of all locally integrable functions $f$ such that
\begin{equation*}
\begin{split}
\big\|f\big\|_{(L^p,L^q)^{\alpha}(\nu,w;\mu)}
:=&\sup_{\ell>0}\left\{\int_{\mathbb R^n}\Big[w(Q(y,\ell))^{1/{\alpha}-1/p-1/q}\big\|f\cdot\chi_{Q(y,\ell)}\big\|_{L^p(\nu)}\Big]^q\mu(y)\,dy\right\}^{1/q}\\
=&\sup_{\ell>0}\Big\|w(Q(y,\ell))^{1/{\alpha}-1/p-1/q}\big\|f\cdot\chi_{Q(y,\ell)}\big\|_{L^p(\nu)}\Big\|_{L^q(\mu)}<\infty,
\end{split}
\end{equation*}
with $w(Q(y,\ell))=\displaystyle\int_{Q(y,\ell)}w(x)\,dx$ and the usual modification when $q=\infty$.
\end{defn}
\begin{defn}
Let $1\leq p\leq\alpha\leq q\leq\infty$, and let $w,\mu$ be two weights on $\mathbb R^n$. We denote by $(WL^p,L^q)^{\alpha}(w;\mu)$ the weighted weak amalgam space consisting of all measurable functions $f$ such that
\begin{equation*}
\begin{split}
\big\|f\big\|_{(WL^p,L^q)^{\alpha}(w;\mu)}
:=&\sup_{\ell>0}\left\{\int_{\mathbb R^n}\Big[w(Q(y,\ell))^{1/{\alpha}-1/p-1/q}\big\|f\cdot\chi_{Q(y,\ell)}\big\|_{WL^p(w)}\Big]^q\mu(y)\,dy\right\}^{1/q}\\
=&\sup_{\ell>0}\Big\|w(Q(y,\ell))^{1/{\alpha}-1/p-1/q}\big\|f\cdot\chi_{Q(y,\ell)}\big\|_{WL^p(w)}\Big\|_{L^q(\mu)}<\infty,
\end{split}
\end{equation*}
with $w(Q(y,\ell))=\displaystyle\int_{Q(y,\ell)}w(x)\,dx$ and the usual modification when $q=\infty$.
\end{defn}
Note that when $\mu\equiv1$, this kind of weighted (weak) amalgam space was introduced by Feuto in \cite{feuto2} (see also \cite{feuto1}). We remark that Feuto considered ball $B$ instead of cube $Q$ in his definition, but these two definitions are apparently equivalent. Also note that when $1\leq p\leq\alpha$ and $q=\infty$, then $(L^p,L^q)^{\alpha}(\nu,w;\mu)$ is just the weighted Morrey space $L^{p,\kappa}(\nu,w)$ with $\kappa=1-p/{\alpha}$, and $(WL^p,L^q)^{\alpha}(w;\mu)$ is just the weighted weak Morrey space $W L^{p,\kappa}(w)$ with $\kappa=1-p/{\alpha}$.

The two-weight problem for classical integral operators has been extensively studied. In \cite{cruz1,cruz2,cruz3} and \cite{liu}, the authors gave some $A_p$-type conditions which are sufficient for the two-weight, weak-type $(p,p)$ inequalities for Calder\'on--Zygmund operators, fractional integral operators, as well as their commutators on the weighted Lebesgue spaces. In \cite{liulz}, the authors established the two-weight, weak-type $(p,p)$ estimates for maximal Bochner--Riesz operators and their commutators. Inspired by the above results, it is natural and interesting to study the weak-type estimates for sublinear operators \eqref{sublinear} and \eqref{frac sublinear}, as well as the corresponding commutators \eqref{sublinear commutator} and \eqref{frac sublinear commutator}.

Let $p'$ be the conjugate index of $p$ whenever $p>1$; that is, $1/p+1/{p'}=1$. The main purpose of this paper
is to investigate the two-weight, weak-type norm inequalities in the setting of weighted Morrey and amalgam spaces. Our main results can be stated as follows.
On the boundedness properties of the sublinear operators and their commutators on weighted Morrey spaces, we will prove
\begin{theorem}\label{mainthm:1}
Let $1<p<\infty$, $0<\kappa<1$ and $\mathcal T$ satisfy the condition $(\ref{sublinear})$. Given a pair of weights $(w,\nu)$, suppose that for some $r>1$ and for all cubes $Q$ in $\mathbb R^n$,
\begin{equation*}
\bigg(\frac{1}{|Q|}\int_Q w(x)^r\,dx\bigg)^{1/{(rp)}}\bigg(\frac{1}{|Q|}\int_Q \nu(x)^{-p'/p}\,dx\bigg)^{1/{p'}}\leq C<\infty.\qquad (\S)
\end{equation*}
Furthermore, we suppose that $\mathcal T$ satisfies the weak-type $(p,p)$ inequality
\begin{equation}\label{assump1.2}
w\big(\big\{x\in\mathbb R^n:\big|\mathcal Tf(x)\big|>\sigma\big\}\big)\leq \frac{C}{\sigma^p}\int_{\mathbb R^n}|f(x)|^p \nu(x)\,dx,
\quad\mbox{for any }~\sigma>0,
\end{equation}
where $C$ does not depend on $f$ and $\sigma>0$. If $w\in\Delta_2$, then the operator $\mathcal T$ is bounded from $L^{p,\kappa}(\nu,w)$ into $WL^{p,\kappa}(w)$.
\end{theorem}

\begin{theorem}\label{mainthm:2}
Let $1<p<\infty$, $0<\kappa<1$ and $\mathcal T_\gamma$ satisfy the condition $(\ref{frac sublinear})$ with $0<\gamma<n$. Given a pair of weights $(w,\nu)$, suppose that for some $r>1$ and for all cubes $Q$ in $\mathbb R^n$,
\begin{equation*}
\big|Q\big|^{\gamma/n}\cdot\bigg(\frac{1}{|Q|}\int_Q w(x)^r\,dx\bigg)^{1/{(rp)}}\bigg(\frac{1}{|Q|}\int_Q \nu(x)^{-p'/p}\,dx\bigg)^{1/{p'}}\leq C<\infty.\qquad(\S')
\end{equation*}
Furthermore, we suppose that $\mathcal T_\gamma$ satisfies the weak-type $(p,p)$ inequality
\begin{equation}\label{assump2.2}
w\big(\big\{x\in\mathbb R^n:\big|\mathcal T_\gamma f(x)\big|>\sigma\big\}\big)
\leq \frac{C}{\sigma^p}\int_{\mathbb R^n}|f(x)|^p \nu(x)\,dx,\quad\mbox{for any }~\sigma>0,
\end{equation}
where $C$ does not depend on $f$ and $\sigma>0$. If $w\in\Delta_2$, then the operator $\mathcal T_\gamma$ is bounded from $L^{p,\kappa}(\nu,w)$ into $WL^{p,\kappa}(w)$.
\end{theorem}

\begin{theorem}\label{mainthm:3}
Let $1<p<\infty$, $0<\kappa<1$, $b\in \mathrm{BMO}(\mathbb R^n)$ and $[b,\mathcal T]$ satisfy the condition $(\ref{sublinear commutator})$. Given a pair of weights $(w,\nu)$, suppose that for some $r>1$ and for all cubes $Q$ in $\mathbb R^n$,
\begin{equation*}
\bigg(\frac{1}{|Q|}\int_Q w(x)^r\,dx\bigg)^{1/{(rp)}}\big\|\nu^{-1/p}\big\|_{\mathcal A,Q}\leq C<\infty,\qquad(\S\S)
\end{equation*}
where $\mathcal A(t)=t^{p'}\big[\log(e+t)\big]^{p'}$. Furthermore, we suppose that $[b,\mathcal T]$ satisfies the weak-type $(p,p)$ inequality
\begin{equation}\label{assump3.2}
w\big(\big\{x\in\mathbb R^n:\big|[b,\mathcal T](f)(x)\big|>\sigma\big\}\big)
\leq\frac{C}{\sigma^p}\int_{\mathbb R^n}|f(x)|^p\nu(x)\,dx,\quad\mbox{for any }~\sigma>0,
\end{equation}
where $C$ does not depend on $f$ and $\sigma>0$. If $w\in A_\infty$, then the commutator operator $[b,\mathcal T]$ is bounded from $L^{p,\kappa}(\nu,w)$ into $WL^{p,\kappa}(w)$.
\end{theorem}

\begin{theorem}\label{mainthm:4}
Let $1<p<\infty$, $0<\kappa<1$, $b\in\mathrm{BMO}(\mathbb R^n)$ and $[b,\mathcal T_\gamma]$ satisfy the condition $(\ref{frac sublinear commutator})$ with $0<\gamma<n$. Given a pair of weights $(w,\nu)$, suppose that for some $r>1$ and for all cubes $Q$ in $\mathbb R^n$,
\begin{equation*}
\big|Q\big|^{\gamma/n}\cdot\bigg(\frac{1}{|Q|}\int_Q w(x)^r\,dx\bigg)^{1/{(rp)}}\big\|\nu^{-1/p}\big\|_{\mathcal A,Q}\leq C<\infty,\qquad(\S\S')
\end{equation*}
where $\mathcal A(t)=t^{p'}\big[\log(e+t)\big]^{p'}$. Furthermore, we suppose that $[b,\mathcal T_\gamma]$ satisfies the weak-type $(p,p)$ inequality
\begin{equation}\label{assump4.2}
w\big(\big\{x\in\mathbb R^n:\big|[b,\mathcal T_\gamma](f)(x)\big|>\sigma\big\}\big)
\leq \frac{C}{\sigma^p}\int_{\mathbb R^n}|f(x)|^p\nu(x)\,dx,\quad\mbox{for any }~\sigma>0,
\end{equation}
where $C$ does not depend on $f$ and $\sigma>0$. If $w\in A_\infty$, then the commutator operator $[b,\mathcal T_\gamma]$ is bounded from $L^{p,\kappa}(\nu,w)$ into $WL^{p,\kappa}(w)$.
\end{theorem}

Concerning the boundedness properties on weighted amalgam spaces for these operators, we have the following results.

\begin{theorem}\label{mainthm:5}
Let $1<p\leq\alpha<q\leq\infty$ and $\mu\in\Delta_2$. Given a pair of weights $(w,\nu)$, assume that for some $r>1$ and for all cubes $Q$ in $\mathbb R^n$,
\begin{equation*}
\bigg(\frac{1}{|Q|}\int_Q w(x)^r\,dx\bigg)^{1/{(rp)}}\bigg(\frac{1}{|Q|}\int_Q \nu(x)^{-p'/p}\,dx\bigg)^{1/{p'}}\leq C<\infty.\qquad (\S)
\end{equation*}
Furthermore, we assume that $\mathcal T$ satisfies the weak-type $(p,p)$ inequality \eqref{assump1.2}. If $w\in \Delta_2$, then the operator $\mathcal T$ is bounded from $(L^p,L^q)^{\alpha}(\nu,w;\mu)$ into $(WL^p,L^q)^{\alpha}(w;\mu)$.
\end{theorem}

\begin{theorem}\label{mainthm:6}
Let $0<\gamma<n$, $1<p\leq\alpha<q\leq\infty$ and $\mu\in\Delta_2$. Given a pair of weights $(w,\nu)$, assume that for some $r>1$ and for all cubes $Q$ in $\mathbb R^n$,
\begin{equation*}
\big|Q\big|^{\gamma/n}\cdot\bigg(\frac{1}{|Q|}\int_Q w(x)^r\,dx\bigg)^{1/{(rp)}}\bigg(\frac{1}{|Q|}\int_Q \nu(x)^{-p'/p}\,dx\bigg)^{1/{p'}}\leq C<\infty.\qquad(\S')
\end{equation*}
Furthermore, we assume that $\mathcal T_\gamma$ satisfies the weak-type $(p,p)$ inequality \eqref{assump2.2}. If $w\in \Delta_2$, then the operator $\mathcal T_{\gamma}$ is bounded from $(L^p,L^q)^{\alpha}(\nu,w;\mu)$ into $(WL^p,L^q)^{\alpha}(w;\mu)$.
\end{theorem}

\begin{theorem}\label{mainthm:7}
Let $1<p\leq\alpha<q\leq\infty$, $\mu\in\Delta_2$ and $b\in \mathrm{BMO}(\mathbb R^n)$. Given a pair of weights $(w,\nu)$, assume that for some $r>1$ and for all cubes $Q$ in $\mathbb R^n$,
\begin{equation*}
\left(\frac{1}{|Q|}\int_Q w(x)^r\,dx\right)^{1/{(rp)}}\big\|\nu^{-1/p}\big\|_{\mathcal A,Q}\leq C<\infty,\qquad(\S\S)
\end{equation*}
where $\mathcal A(t)=t^{p'}\big[\log(e+t)\big]^{p'}$. Furthermore, we assume that $[b,\mathcal T]$ satisfies the weak-type $(p,p)$ inequality \eqref{assump3.2}. If $w\in A_\infty$, then the commutator operator $[b,\mathcal T]$ is bounded from $(L^p,L^q)^{\alpha}(\nu,w;\mu)$ into $(WL^p,L^q)^{\alpha}(w;\mu)$.
\end{theorem}

\begin{theorem}\label{mainthm:8}
Let $0<\gamma<n$, $1<p\leq\alpha<q\leq\infty$, $\mu\in\Delta_2$ and $b\in\mathrm{BMO}(\mathbb R^n)$. Given a pair of weights $(w,\nu)$, assume that for some $r>1$ and for all cubes $Q$ in $\mathbb R^n$,
\begin{equation*}
\big|Q\big|^{\gamma/n}\cdot\bigg(\frac{1}{|Q|}\int_Q w(x)^r\,dx\bigg)^{1/{(rp)}}\big\|\nu^{-1/p}\big\|_{\mathcal A,Q}\leq C<\infty,\qquad(\S\S')
\end{equation*}
where $\mathcal A(t)=t^{p'}\big[\log(e+t)\big]^{p'}$. Furthermore, we assume that $[b,\mathcal T_\gamma]$ satisfies the weak-type $(p,p)$ inequality \eqref{assump4.2}. If $w\in A_\infty$, then the commutator operator $[b,\mathcal T_{\gamma}]$ is bounded from $(L^p,L^q)^{\alpha}(\nu,w;\mu)$ into $(WL^p,L^q)^{\alpha}(w;\mu)$.
\end{theorem}
\newtheorem{remark}[theorem]{Remark}
\begin{remark}
It should be pointed out that the conclusions of our main theorems are natural generalizations of the corresponding weak-type estimates on the weighted Lebesgue spaces. The operators satisfying the assumptions of the above theorems include Calder\'on--Zygmund operators, Bochner--Riesz operators, and fractional integral operators. Hence, we are able to apply our main theorems to these classical integral operators.
\end{remark}

\section{Notations and definitions}

\subsection{Weights}
A non-negative function $w$ defined on $\mathbb R^n$ will be called a weight if it is locally integrable. $Q(x_0,\ell)$ will denote the cube centered at $x_0$ and has side length $\ell>0$, all cubes are assumed to have their sides parallel to the coordinate axes. Given a cube $Q=Q(x_0,\ell)$ and $\lambda>0$, $\lambda Q$ stands for the cube concentric with $Q$ and having side length $\lambda\sqrt{n}$ times as long, i.e., $\lambda Q(x_0,\ell):=Q(x_0,\lambda\sqrt{n}\ell)$. For any given weight $w$ and any Lebesgue measurable set $E$ of $\mathbb R^n$, we denote the characteristic function of $E$ by $\chi_E$, the Lebesgue measure of $E$ by $|E|$ and the weighted measure of $E$ by $w(E)$, where $w(E):=\displaystyle\int_E w(x)\,dx$. We also denote $E^c:=\mathbb R^n\backslash E$ the complement of $E$. Given a weight $w$, we say that $w$ satisfies the \emph{doubling condition}, if there exists a universal constant $C>0$ such that for any cube $Q\subset\mathbb R^n$, we have
\begin{equation}\label{weights}
w(2Q)\leq C\cdot w(Q).
\end{equation}
When $w$ satisfies this condition \eqref{weights}, we denote $w\in\Delta_2$ for brevity. A weight $w$ is said to belong to the Muckenhoupt's class $A_p$ for $1<p<\infty$, if there exists a constant $C>0$ such that
\begin{equation*}
\left(\frac1{|Q|}\int_Q w(x)\,dx\right)^{1/p}\left(\frac1{|Q|}\int_Q w(x)^{-p'/p}\,dx\right)^{1/{p'}}\leq C
\end{equation*}
holds for every cube $Q\subset\mathbb R^n$. The class $A_{\infty}$ is defined as the union of the $A_p$ classes for $1<p<\infty$, i.e., $A_\infty=\bigcup_{1<p<\infty}A_p$. If $w$ is an $A_{\infty}$ weight, then we have $w\in\Delta_2$ (see \cite{garcia}).
Moreover, this class $A_\infty$ is characterized as the class of all weights satisfying the following property: there exists a number $\delta>0$ and a finite constant $C>0$ such that (see \cite{garcia})
\begin{equation}\label{compare}
\frac{w(E)}{w(Q)}\le C\left(\frac{|E|}{|Q|}\right)^\delta
\end{equation}
holds for every cube $Q\subset\mathbb R^n$ and all measurable subsets $E$ of $Q$. Given a weight $w$ on $\mathbb R^n$ and for $1\leq p<\infty$, the weighted Lebesgue space $L^p(w)$ is defined as the set of all measurable functions $f$ such that
\begin{equation*}
\big\|f\big\|_{L^p(w)}:=\bigg(\int_{\mathbb R^n}|f(x)|^pw(x)\,dx\bigg)^{1/p}<\infty.
\end{equation*}
We also define the weighted weak Lebesgue space $WL^p(w)$ ($1\leq p<\infty$) as the set of all measurable functions $f$ satisfying
\begin{equation*}
\big\|f\big\|_{WL^p(w)}:=
\sup_{\sigma>0}\sigma\cdot\Big[w\big(\big\{x\in\mathbb R^n:|f(x)|>\sigma\big\}\big)\Big]^{1/p}<\infty.
\end{equation*}

\subsection{Orlicz spaces and BMO}
We next recall some basic facts about Orlicz spaces needed for the proofs of the main results. For further details, we refer the reader to \cite{rao}. A function $\mathcal A:[0,+\infty)\rightarrow[0,+\infty)$ is said to be a Young function if it is continuous, convex and strictly increasing satisfying $\mathcal A(0)=0$ and $\mathcal A(t)\to +\infty$ as $t\to +\infty$. Given a Young function $\mathcal A$, we define the $\mathcal A$-average of a function $f$ over a cube $Q$ by means of the following Luxemburg norm:
\begin{equation*}
\big\|f\big\|_{\mathcal A,Q}
:=\inf\left\{\lambda>0:\frac{1}{|Q|}\int_Q\mathcal A\left(\frac{|f(x)|}{\lambda}\right)dx\leq1\right\}.
\end{equation*}
In particular, when $\mathcal A(t)=t^p$, $1<p<\infty$, it is easy to check that
\begin{equation}\label{norm}
\big\|f\big\|_{\mathcal A,Q}=\left(\frac{1}{|Q|}\int_Q|f(x)|^p\,dx\right)^{1/p};
\end{equation}
that is, the Luxemburg norm coincides with the normalized $L^p$ norm. The main examples that we are going to consider are $\mathcal A(t)=t^p\big[\log(e+t)\big]^p$ with $1<p<\infty$.

Let us now recall the definition of the space of $\mathrm{BMO}(\mathbb R^n)$ (see \cite{john}). A locally integrable function $b$ is said to be in $\mathrm{BMO}(\mathbb R^n)$, if
\begin{equation*}
\|b\|_*:=\sup_{Q\subset\mathbb R^n}\frac{1}{|Q|}\int_Q|b(x)-b_Q|\,dx<\infty,
\end{equation*}
where $b_Q$ denotes the mean value of $b$ over $Q$, namely,
\begin{equation*}
b_Q:=\frac{1}{|Q|}\int_Q b(y)\,dy
\end{equation*}
and the supremum is taken over all cubes $Q$ in $\mathbb R^n$. Modulo constants, the space $\mathrm{BMO}(\mathbb R^n)$ is a Banach space with respect to the norm $\|\cdot\|_*$.

Throughout this paper $C$ always denotes a positive constant independent of the main parameters involved, but it may be different from line to line. We will use $c_1,\dots, c_4$ appearing in the first section of this paper to denote certain constants. We also use $A\approx B$ to denote the equivalence of $A$ and $B$; that is, there exist two positive constants $C_1$, $C_2$ independent of $A$, $B$ such that $C_1 A\leq B\leq C_2 A$.

\section{Proofs of Theorems \ref{mainthm:1} and \ref{mainthm:2}}

\begin{proof}[Proof of Theorem $\ref{mainthm:1}$]
Let $f\in L^{p,\kappa}(\nu,w)$ with $1<p<\infty$ and $0<\kappa<1$. For an arbitrary fixed cube $Q=Q(x_0,\ell)\subset\mathbb R^n$, we set $2Q:=Q(x_0,2\sqrt{n}\ell)$. Decompose $f$ as
\begin{equation*}
\begin{cases}
f=f_1+f_2\in L^{p,\kappa}(\nu,w);\  &\\
f_1=f\cdot\chi_{2Q};\  &\\
f_2=f\cdot\chi_{(2Q)^c},
\end{cases}
\end{equation*}
where $\chi_{E}$ denotes the characteristic function of the set $E$. Then for any given $\sigma>0$, we write
\begin{equation*}
\begin{split}
&\frac{1}{w(Q)^{\kappa/p}}\sigma\cdot \Big[w\big(\big\{x\in Q:\big|\mathcal T(f)(x)\big|>\sigma\big\}\big)\Big]^{1/p}\\
\leq &\frac{1}{w(Q)^{\kappa/p}}\sigma\cdot \Big[w\big(\big\{x\in Q:\big|\mathcal T(f_1)(x)\big|>\sigma/2\big\}\big)\Big]^{1/p}\\
&+\frac{1}{w(Q)^{\kappa/p}}\sigma\cdot \Big[w\big(\big\{x\in Q:\big|\mathcal T(f_2)(x)\big|>\sigma/2\big\}\big)\Big]^{1/p}\\
:=&I_1+I_2.
\end{split}
\end{equation*}
We first consider the term $I_1$. Using the assumption \eqref{assump1.2} and the condition $w\in\Delta_2$, we get
\begin{equation*}
\begin{split}
I_1&\leq C\cdot\frac{1}{w(Q)^{\kappa/p}}\bigg(\int_{\mathbb R^n}|f_1(x)|^p\nu(x)\,dx\bigg)^{1/p}\\
&=C\cdot\frac{1}{w(Q)^{\kappa/p}}\bigg(\int_{2Q}|f(x)|^p\nu(x)\,dx\bigg)^{1/p}\\
&\leq C\big\|f\big\|_{L^{p,\kappa}(\nu,w)}\cdot\frac{w(2Q)^{\kappa/p}}{w(Q)^{\kappa/p}}\\
&\leq C\big\|f\big\|_{L^{p,\kappa}(\nu,w)}.
\end{split}
\end{equation*}
This is just our desired estimate. Let us estimate the second term $I_2$. To this end, we observe that when $x\in Q$ and $y\in(2Q)^c$, one has $|x-y|\approx|x_0-y|$. We then decompose $\mathbb R^n$ into a geometrically increasing sequence of concentric cubes, and obtain the following pointwise estimate by the condition \eqref{sublinear}.
\begin{align}\label{pointwise1}
\big|\mathcal T(f_2)(x)\big|&\leq c_1\int_{\mathbb R^n}\frac{|f_2(y)|}{|x-y|^n}dy
\leq C\int_{(2Q)^c}\frac{|f(y)|}{|x_0-y|^n}dy\notag\\
&=C\sum_{j=1}^\infty\int_{2^{j+1}Q\backslash 2^{j}Q}\frac{|f(y)|}{|x_0-y|^n}dy\notag\\
&\leq C\sum_{j=1}^\infty\frac{1}{|2^{j+1}Q|}\int_{2^{j+1}Q}|f(y)|\,dy.
\end{align}
This pointwise estimate \eqref{pointwise1} together with Chebyshev's inequality implies that
\begin{equation*}
\begin{split}
I_2&\leq \frac{2}{w(Q)^{\kappa/p}}\cdot\left(\int_Q\big|\mathcal T(f_2)(x)\big|^pw(x)\,dx\right)^{1/p}\\
&\leq C\cdot w(Q)^{{(1-\kappa)}/p}\sum_{j=1}^\infty\frac{1}{|2^{j+1}Q|}\int_{2^{j+1}Q}|f(y)|\,dy.
\end{split}
\end{equation*}
It follows directly from H\"older's inequality with exponent $p>1$ that
\begin{equation*}
\begin{split}
I_2&\leq C\cdot w(Q)^{{(1-\kappa)}/p}
\sum_{j=1}^\infty\frac{1}{|2^{j+1}Q|}\left(\int_{2^{j+1}Q}|f(y)|^p\nu(y)\,dy\right)^{1/p}\\
&\times\left(\int_{2^{j+1}Q}\nu(y)^{-p'/p}\,dy\right)^{1/{p'}}\\
&\leq C\big\|f\big\|_{L^{p,\kappa}(\nu,w)}\cdot w(Q)^{{(1-\kappa)}/p}\sum_{j=1}^\infty\frac{w(2^{j+1}Q)^{\kappa/p}}{|2^{j+1}Q|}
\times\left(\int_{2^{j+1}Q}\nu(y)^{-p'/p}\,dy\right)^{1/{p'}}\\
&=C\big\|f\big\|_{L^{p,\kappa}(\nu,w)}\sum_{j=1}^\infty\frac{w(Q)^{{(1-\kappa)}/p}}{w(2^{j+1}Q)^{{(1-\kappa)}/p}}
\cdot\frac{w(2^{j+1}Q)^{1/p}}{|2^{j+1}Q|}
\times\left(\int_{2^{j+1}Q}\nu(y)^{-p'/p}\,dy\right)^{1/{p'}}.
\end{split}
\end{equation*}
Moreover, for any positive integer $j$, we apply H\"older's inequality again with exponent $r>1$ to get
\begin{align}\label{U}
w\big(2^{j+1}Q\big)&=\int_{2^{j+1}Q}w(y)\,dy\notag\\
&\leq\big|2^{j+1}Q\big|^{1/{r'}}\bigg(\int_{2^{j+1}Q}w(y)^r\,dy\bigg)^{1/r}.
\end{align}
Thus, in view of \eqref{U}, we conclude that
\begin{equation*}
\begin{split}
I_2&\leq C\big\|f\big\|_{L^{p,\kappa}(\nu,w)}\sum_{j=1}^\infty\frac{w(Q)^{{(1-\kappa)}/p}}{w(2^{j+1}Q)^{{(1-\kappa)}/p}}
\cdot\frac{|2^{j+1}Q|^{1/{(r'p)}}}{|2^{j+1}Q|}\\
&\times\left(\int_{2^{j+1}Q}w(y)^r\,dy\right)^{1/{(rp)}}\left(\int_{2^{j+1}Q}\nu(y)^{-p'/p}\,dy\right)^{1/{p'}}\\
&\leq C\big\|f\big\|_{L^{p,\kappa}(\nu,w)}
\times\sum_{j=1}^\infty\frac{w(Q)^{{(1-\kappa)}/p}}{w(2^{j+1}Q)^{{(1-\kappa)}/p}}.
\end{split}
\end{equation*}
The last inequality is obtained by the $A_p$-type condition ($\S$) on $(w,\nu)$. Since $w\in\Delta_2$, we can easily check that there exists a \emph{reverse doubling constant} $D=D(w)>1$ independent of $Q$ such that (see \cite[Lemma 4.1]{komori})
\begin{equation*}
w(2Q)\geq D\cdot w(Q), \quad \mbox{for all cubes }\,Q\subset\mathbb R^n,
\end{equation*}
which implies that for any positive integer $j$,
\begin{equation}\label{iteration}
w\big(2^{j+1}Q\big)\geq D^{j+1}\cdot w(Q)
\end{equation}
by iteration. Hence,
\begin{align}\label{C1}
\sum_{j=1}^\infty\frac{w(Q)^{{(1-\kappa)}/p}}{w(2^{j+1}Q)^{{(1-\kappa)}/p}}
&\leq \sum_{j=1}^\infty\left(\frac{w(Q)}{D^{j+1}\cdot w(Q)}\right)^{{(1-\kappa)}/p}\notag\\
&=\sum_{j=1}^\infty\left(\frac{1}{D^{j+1}}\right)^{{(1-\kappa)}/p}\notag\\
&\leq C,
\end{align}
where the last series is convergent since the \emph{reverse doubling constant} $D>1$ and $0<\kappa<1$. This yields our desired estimate
\begin{equation*}
I_2\leq C\big\|f\big\|_{L^{p,\kappa}(\nu,w)}.
\end{equation*}
Summing up the above estimates for $I_1$ and $I_2$, and then taking the supremum over all cubes $Q\subset\mathbb R^n$ and all $\sigma>0$, we finish the proof of Theorem \ref{mainthm:1}.
\end{proof}

\begin{proof}[Proof of Theorem $\ref{mainthm:2}$]
Let $f\in L^{p,\kappa}(\nu,w)$ with $1<p<\infty$ and $0<\kappa<1$. For an arbitrary fixed cube $Q=Q(x_0,\ell)$ in $\mathbb R^n$, we decompose $f$ as
\begin{equation*}
\begin{cases}
f=f_1+f_2\in L^{p,\kappa}(\nu,w);\  &\\
f_1=f\cdot\chi_{2Q};\  &\\
f_2=f\cdot\chi_{(2Q)^c},
\end{cases}
\end{equation*}
where $2Q:=Q(x_0,2\sqrt{n}\ell)$. For any given $\sigma>0$, we then write
\begin{equation*}
\begin{split}
&\frac{1}{w(Q)^{\kappa/p}}\sigma\cdot \Big[w\big(\big\{x\in Q:\big|\mathcal T_\gamma(f)(x)\big|>\sigma\big\}\big)\Big]^{1/p}\\
\leq &\frac{1}{w(Q)^{\kappa/p}}\sigma\cdot \Big[w\big(\big\{x\in Q:\big|\mathcal T_\gamma(f_1)(x)\big|>\sigma/2\big\}\big)\Big]^{1/p}\\
&+\frac{1}{w(Q)^{\kappa/p}}\sigma\cdot \Big[w\big(\big\{x\in Q:\big|\mathcal T_\gamma(f_2)(x)\big|>\sigma/2\big\}\big)\Big]^{1/p}\\
:=&I'_1+I'_2.
\end{split}
\end{equation*}
Let us consider the first term $I'_1$. Using the assumption \eqref{assump2.2} and the condition $w\in\Delta_2$, we have
\begin{equation*}
\begin{split}
I'_1&\leq C\cdot\frac{1}{w(Q)^{\kappa/p}}\left(\int_{\mathbb R^n}|f_1(x)|^p\nu(x)\,dx\right)^{1/p}\\
&=C\cdot\frac{1}{w(Q)^{\kappa/p}}\left(\int_{2Q}|f(x)|^p\nu(x)\,dx\right)^{1/p}\\
&\leq C\big\|f\big\|_{L^{p,\kappa}(\nu,w)}\cdot\frac{w(2Q)^{\kappa/p}}{w(Q)^{\kappa/p}}\\
&\leq C\big\|f\big\|_{L^{p,\kappa}(\nu,w)}.
\end{split}
\end{equation*}
This is exactly what we want. We now deal with the second term $I'_2$. Note that $|x-y|\approx|x_0-y|$, whenever $x,x_0\in Q$ and $y\in(2Q)^c$. For $0<\gamma<n$ and all $x\in Q$, using the standard technique and the condition \eqref{frac sublinear}, we can see that
\begin{align}\label{pointwise2}
\big|\mathcal T_{\gamma}(f_2)(x)\big|&\leq c_2\int_{\mathbb R^n}\frac{|f_2(y)|}{|x-y|^{n-\gamma}}dy
\leq C\int_{(2Q)^c}\frac{|f(y)|}{|x_0-y|^{n-\gamma}}dy\notag\\
&=C\sum_{j=1}^\infty\int_{2^{j+1}Q\backslash 2^{j}Q}\frac{|f(y)|}{|x_0-y|^{n-\gamma}}dy\notag\\
&\leq C\sum_{j=1}^\infty\frac{1}{|2^{j+1}Q|^{1-\gamma/n}}\int_{2^{j+1}Q}|f(y)|\,dy.
\end{align}
This pointwise estimate \eqref{pointwise2} together with Chebyshev's inequality yields
\begin{equation*}
\begin{split}
I'_2&\leq \frac{2}{w(Q)^{\kappa/p}}\cdot\left(\int_Q\big|\mathcal T_\gamma(f_2)(x)\big|^pw(x)\,dx\right)^{1/p}\\
&\leq C\cdot w(Q)^{{(1-\kappa)}/p}\sum_{j=1}^\infty\frac{1}{|2^{j+1}Q|^{1-\gamma/n}}\int_{2^{j+1}Q}|f(y)|\,dy.
\end{split}
\end{equation*}
By using H\"older's inequality with exponent $p>1$, we can deduce that
\begin{equation*}
\begin{split}
I'_2&\leq C\cdot w(Q)^{{(1-\kappa)}/p}
\sum_{j=1}^\infty\frac{1}{|2^{j+1}Q|^{1-\gamma/n}}\left(\int_{2^{j+1}Q}|f(y)|^p\nu(y)\,dy\right)^{1/p}\\
&\times\left(\int_{2^{j+1}Q}\nu(y)^{-p'/p}\,dy\right)^{1/{p'}}\\
&\leq C\big\|f\big\|_{L^{p,\kappa}(\nu,w)}\cdot w(Q)^{{(1-\kappa)}/p}\sum_{j=1}^\infty\frac{w(2^{j+1}Q)^{\kappa/p}}{|2^{j+1}Q|^{1-\gamma/n}}
\times\left(\int_{2^{j+1}Q}\nu(y)^{-p'/p}\,dy\right)^{1/{p'}}\\
&=C\big\|f\big\|_{L^{p,\kappa}(\nu,w)}\sum_{j=1}^\infty\frac{w(Q)^{{(1-\kappa)}/p}}{w(2^{j+1}Q)^{{(1-\kappa)}/p}}
\cdot\frac{w(2^{j+1}Q)^{1/p}}{|2^{j+1}Q|^{1-\gamma/n}}
\times\left(\int_{2^{j+1}Q}\nu(y)^{-p'/p}\,dy\right)^{1/{p'}}.
\end{split}
\end{equation*}
Moreover, we apply the estimate \eqref{U} to get
\begin{equation*}
\begin{split}
I'_2&\leq C\big\|f\big\|_{L^{p,\kappa}(\nu,w)}\sum_{j=1}^\infty\frac{w(Q)^{{(1-\kappa)}/p}}{w(2^{j+1}Q)^{{(1-\kappa)}/p}}
\cdot\frac{|2^{j+1}Q|^{1/{(r'p)}}}{|2^{j+1}Q|}\\
&\times\big|2^{j+1}Q\big|^{\gamma/n}\cdot
\left(\int_{2^{j+1}Q}w(y)^r\,dy\right)^{1/{(rp)}}\left(\int_{2^{j+1}Q}\nu(y)^{-p'/p}\,dy\right)^{1/{p'}}\\
&\leq C\big\|f\big\|_{L^{p,\kappa}(\nu,w)}
\times\sum_{j=1}^\infty\frac{w(Q)^{{(1-\kappa)}/p}}{w(2^{j+1}Q)^{{(1-\kappa)}/p}}.
\end{split}
\end{equation*}
The last inequality is obtained by the $A_p$-type condition $(\S')$ on $(w,\nu)$. Therefore, in view of \eqref{C1}, we find that
\begin{equation*}
I'_2\leq C\big\|f\big\|_{L^{p,\kappa}(\nu,w)}.
\end{equation*}
Combining the above estimates for $I'_1$ and $I'_2$, and then taking the supremum over all cubes $Q\subset\mathbb R^n$ and all $\sigma>0$, we complete the proof of Theorem \ref{mainthm:2}.
\end{proof}

\section{Proofs of Theorems \ref{mainthm:3} and \ref{mainthm:4}}

For the results involving commutators, we need the following properties of $\mathrm{BMO}(\mathbb R^n)$, which can be found in \cite{perez1}.
\begin{lemma}\label{BMO}
Let $b$ be a function in $\mathrm{BMO}(\mathbb R^n)$.

$(i)$ For every cube $Q$ in $\mathbb R^n$ and for any positive integer $j$, then
\begin{equation*}
\big|b_{2^{j+1}Q}-b_Q\big|\leq C\cdot(j+1)\|b\|_*.
\end{equation*}

$(ii)$ Let $1<p<\infty$. For every cube $Q$ in $\mathbb R^n$ and for any $w\in A_{\infty}$, then
\begin{equation*}
\bigg(\int_Q\big|b(x)-b_Q\big|^pw(x)\,dx\bigg)^{1/p}\leq C\|b\|_*\cdot w(Q)^{1/p}.
\end{equation*}
\end{lemma}

Before proving our main theorems, we will also need a generalization of H\"older's inequality due to O'Neil \cite{neil}.

\begin{lemma}\label{three}
Let $\mathcal A$, $\mathcal B$ and $\mathcal C$ be Young functions such that for all $t>0$,
\begin{equation*}
\mathcal A^{-1}(t)\cdot\mathcal B^{-1}(t)\leq\mathcal C^{-1}(t),
\end{equation*}
where $\mathcal A^{-1}(t)$ is the inverse function of $\mathcal A(t)$. Then for all functions $f$ and $g$ and all cubes $Q$ in $\mathbb R^n$,
\begin{equation*}
\big\|f\cdot g\big\|_{\mathcal C,Q}\leq 2\big\|f\big\|_{\mathcal A,Q}\big\|g\big\|_{\mathcal B,Q}.
\end{equation*}
\end{lemma}

We are now ready to give the proofs of Theorems \ref{mainthm:3} and \ref{mainthm:4}.

\begin{proof}[Proof of Theorem $\ref{mainthm:3}$]
Let $f\in L^{p,\kappa}(\nu,w)$ with $1<p<\infty$ and $0<\kappa<1$. For any given cube $Q=Q(x_0,\ell)\subset\mathbb R^n$, we split $f$ as usual by
\begin{equation*}
f=f\cdot\chi_{2Q}+f\cdot\chi_{(2Q)^c}:=f_1+f_2,
\end{equation*}
where $2Q:=Q(x_0,2\sqrt{n}\ell)$. Then for any given $\sigma>0$, one writes
\begin{equation*}
\begin{split}
&\frac{1}{w(Q)^{\kappa/p}}\sigma\cdot \Big[w\big(\big\{x\in Q:\big|[b,\mathcal T](f)(x)\big|>\sigma\big\}\big)\Big]^{1/p}\\
\leq &\frac{1}{w(Q)^{\kappa/p}}\sigma\cdot \Big[w\big(\big\{x\in Q:\big|[b,\mathcal T](f_1)(x)\big|>\sigma/2\big\}\big)\Big]^{1/p}\\
&+\frac{1}{w(Q)^{\kappa/p}}\sigma\cdot \Big[w\big(\big\{x\in Q:\big|[b,\mathcal T](f_2)(x)\big|>\sigma/2\big\}\big)\Big]^{1/p}\\
:=&J_1+J_2.
\end{split}
\end{equation*}
Since $w$ is an $A_{\infty}$ weight, we know that $w\in\Delta_2$. By our assumption \eqref{assump3.2} and inequality (\ref{weights}), we have
\begin{equation*}
\begin{split}
J_1&\leq C\cdot\frac{1}{w(Q)^{\kappa/p}}\left(\int_{\mathbb R^n}|f_1(x)|^p\nu(x)\,dx\right)^{1/p}\\
&=C\cdot\frac{1}{w(Q)^{\kappa/p}}\left(\int_{2Q}|f(x)|^p\nu(x)\,dx\right)^{1/p}\\
&\leq C\big\|f\big\|_{L^{p,\kappa}(\nu,w)}\cdot\frac{w(2Q)^{\kappa/p}}{w(Q)^{\kappa/p}}\\
&\leq C\big\|f\big\|_{L^{p,\kappa}(\nu,w)},
\end{split}
\end{equation*}
which is exactly what we want. For any $x\in Q$, from the size condition \eqref{sublinear commutator}, we can easily see that
\begin{equation*}
\begin{split}
\big|[b,\mathcal T](f_2)(x)\big|&\leq c_3\int_{\mathbb R^n}\frac{|b(x)-b(y)|\cdot|f_2(y)|}{|x-y|^n}dy\\
&\leq c_3\big|b(x)-b_{Q}\big|\cdot\int_{\mathbb R^n}\frac{|f_2(y)|}{|x-y|^n}dy+c_3\int_{\mathbb R^n}\frac{|b(y)-b_{Q}|\cdot|f_2(y)|}{|x-y|^n}dy\\
&:=\xi(x)+\eta(x).
\end{split}
\end{equation*}
Hence, we can further split $J_2$ into two parts as follows:
\begin{equation*}
\begin{split}
J_2\leq&\frac{1}{w(Q)^{\kappa/p}}\sigma\cdot\Big[w\big(\big\{x\in Q:\xi(x)>\sigma/4\big\}\big)\Big]^{1/p}
+\frac{1}{w(Q)^{\kappa/p}}\sigma\cdot\Big[w\big(\big\{x\in Q:\eta(x)>\sigma/4\big\}\big)\Big]^{1/p}\\
:=&J_3+J_4.
\end{split}
\end{equation*}
For the term $J_3$, it follows from the pointwise estimate \eqref{pointwise1} and Chebyshev's inequality that
\begin{equation*}
\begin{split}
J_3&\leq\frac{4}{w(Q)^{\kappa/p}}\cdot\left(\int_Q |\xi(x)|^pw(x)\,dx\right)^{1/p}\\
&\leq\frac{C}{w(Q)^{\kappa/p}}\cdot\left(\int_Q \big|b(x)-b_{Q}\big|^pw(x)\,dx\right)^{1/p}\sum_{j=1}^\infty\frac{1}{|2^{j+1}Q|}\int_{2^{j+1}Q}|f(y)|\,dy\\
&\leq C\|b\|_*\cdot w(Q)^{{(1-\kappa)}/p}\sum_{j=1}^\infty\frac{1}{|2^{j+1}Q|}\int_{2^{j+1}Q}|f(y)|\,dy,
\end{split}
\end{equation*}
where in the last inequality we have used the second part of Lemma \ref{BMO} since $w\in A_{\infty}$. Repeating the arguments used in Theorem \ref{mainthm:1}, we can also prove that
\begin{equation*}
J_3\leq C\big\|f\big\|_{L^{p,\kappa}(\nu,w)}.
\end{equation*}
Let us consider the term $J_4$. Similar to the proof of \eqref{pointwise1}, for any given $x\in Q$, we can obtain the following pointwise estimate as well.
\begin{align}\label{pointwise3}
\big|\eta(x)\big|
&= c_3\int_{\mathbb R^n}\frac{|b(y)-b_{Q}|\cdot|f_2(y)|}{|x-y|^n}dy\notag\\
&\leq C\int_{(2Q)^c}\frac{|b(y)-b_{Q}|\cdot|f(y)|}{|x_0-y|^n}dy\notag\\
&\leq C\sum_{j=1}^\infty\frac{1}{|2^{j+1}Q|}\int_{2^{j+1}Q}\big|b(y)-b_{Q}\big|\cdot\big|f(y)\big|\,dy.
\end{align}
This, together with Chebyshev's inequality, yields
\begin{equation*}
\begin{split}
J_4&\leq\frac{4}{w(Q)^{\kappa/p}}\cdot\left(\int_Q|\eta(x)|^pw(x)\,dx\right)^{1/p}\\
&\leq C\cdot w(Q)^{{(1-\kappa)}/p}\cdot
\sum_{j=1}^\infty\frac{1}{|2^{j+1}Q|}\int_{2^{j+1}Q}\big|b(y)-b_{Q}\big|\cdot\big|f(y)\big|\,dy\\
&\leq C\cdot w(Q)^{{(1-\kappa)}/p}\cdot
\sum_{j=1}^\infty\frac{1}{|2^{j+1}Q|}\int_{2^{j+1}Q}\big|b(y)-b_{{2^{j+1}Q}}\big|\cdot\big|f(y)\big|\,dy\\
&+C\cdot w(Q)^{{(1-\kappa)}/p}\cdot
\sum_{j=1}^\infty\frac{1}{|2^{j+1}Q|}\int_{2^{j+1}Q}\big|b_{{2^{j+1}Q}}-b_{Q}\big|\cdot\big|f(y)\big|\,dy\\
&:=J_5+J_6.
\end{split}
\end{equation*}
An application of H\"older's inequality leads to that
\begin{equation*}
\begin{split}
J_5&\leq C\cdot w(Q)^{{(1-\kappa)}/p}\cdot\sum_{j=1}^\infty\frac{1}{|2^{j+1}Q|}
\bigg(\int_{2^{j+1}Q}|f(y)|^p\nu(y)\,dy\bigg)^{1/p}\\
&\times\bigg(\int_{2^{j+1}Q}\big|b(y)-b_{{2^{j+1}Q}}\big|^{p'}\nu(y)^{-p'/p}\,dy\bigg)^{1/{p'}}\\
&\leq C\big\|f\big\|_{L^{p,\kappa}(\nu,w)}\cdot w(Q)^{{(1-\kappa)}/p}\sum_{j=1}^\infty\frac{w(2^{j+1}Q)^{\kappa/p}}{|2^{j+1}Q|}\\
&\times\big|2^{j+1}Q\big|^{1/{p'}}\Big\|\big[b-b_{{2^{j+1}Q}}\big]\cdot \nu^{-1/p}\Big\|_{\mathcal C,2^{j+1}Q},
\end{split}
\end{equation*}
where $\mathcal C(t)=t^{p'}$ is a Young function by \eqref{norm}. For $1<p<\infty$, it is immediate that the inverse function of $\mathcal C(t)$ is $\mathcal C^{-1}(t)=t^{1/{p'}}$. Also observe that the following identity is true:
\begin{equation*}
\begin{split}
\mathcal C^{-1}(t)&=t^{1/{p'}}\\
&=\frac{t^{1/{p'}}}{\log(e+t)}\times\log(e+t)\\
&=\mathcal A^{-1}(t)\cdot\mathcal B^{-1}(t),
\end{split}
\end{equation*}
where
\begin{equation*}
\mathcal A(t)\approx t^{p'}\big[\log(e+t)\big]^{p'}\qquad \&\qquad \mathcal B(t)\approx \exp(t)-1.
\end{equation*}
Let $\|h\|_{\exp L,Q}$ denote the mean Luxemburg norm of $h$ on cube $Q$ with Young function $\mathcal B(t)\approx \exp(t)-1$. Thus, by Lemma \ref{three}, we have
\begin{align}\label{key}
\Big\|\big[b-b_{{2^{j+1}Q}}\big]\cdot \nu^{-1/p}\Big\|_{\mathcal C,2^{j+1}Q}
&\leq C\big\|b-b_{{2^{j+1}Q}}\big\|_{\exp L,2^{j+1}Q}\cdot\big\|\nu^{-1/p}\big\|_{\mathcal A,2^{j+1}Q}\notag\\
&\leq C\|b\|_*\cdot\big\|\nu^{-1/p}\big\|_{\mathcal A,2^{j+1}Q},
\end{align}
where in the last inequality we have used the well-known fact that (see \cite{perez1})
\begin{equation}\label{Jensen}
\big\|b-b_{Q}\big\|_{\exp L,Q}\leq C\|b\|_*,\qquad \mbox{for any cube }Q\subset\mathbb R^n.
\end{equation}
This is equivalent to the following inequality
\begin{equation*}
\frac{1}{|Q|}\int_Q\exp\bigg(\frac{|b(y)-b_Q|}{c_0\|b\|_*}\bigg)\,dy\leq C,\qquad \mbox{for any cube }Q\subset\mathbb R^n,
\end{equation*}
which is just a corollary of the celebrated John--Nirenberg's inequality (see \cite{john}). Consequently,
\begin{equation*}
J_5\leq C\|b\|_*\big\|f\big\|_{L^{p,\kappa}(\nu,w)}
\sum_{j=1}^\infty\frac{w(Q)^{{(1-\kappa)}/p}}{w(2^{j+1}Q)^{{(1-\kappa)}/p}}
\cdot\frac{w(2^{j+1}Q)^{1/p}}{|2^{j+1}Q|^{1/p}}\cdot\big\|\nu^{-1/p}\big\|_{\mathcal A,2^{j+1}Q}.
\end{equation*}
Moreover, in view of \eqref{U}, we can deduce that
\begin{equation*}
\begin{split}
J_5&\leq C\big\|f\big\|_{L^{p,\kappa}(\nu,w)}\sum_{j=1}^\infty\frac{w(Q)^{{(1-\kappa)}/p}}{w(2^{j+1}Q)^{{(1-\kappa)}/p}}\\
&\times\left(\frac{1}{|2^{j+1}Q|}\int_{2^{j+1}Q}w(x)^r\,dx\right)^{1/{(rp)}}\cdot\big\|\nu^{-1/p}\big\|_{\mathcal A,2^{j+1}Q}\\
&\leq C\big\|f\big\|_{L^{p,\kappa}(\nu,w)}\times\sum_{j=1}^\infty\frac{w(Q)^{{(1-\kappa)}/p}}{w(2^{j+1}Q)^{{(1-\kappa)}/p}}\\
&\leq C\big\|f\big\|_{L^{p,\kappa}(\nu,w)}.
\end{split}
\end{equation*}
The last inequality is obtained by the $A_p$-type condition $(\S\S)$ on $(w,\nu)$ and the estimate \eqref{C1}. It remains to estimate the last term $J_6$. Applying the first part of Lemma \ref{BMO} and H\"older's inequality, we get
\begin{equation*}
\begin{split}
J_6&\leq C\cdot w(Q)^{{(1-\kappa)}/p}\cdot
\sum_{j=1}^\infty\frac{(j+1)\|b\|_*}{|2^{j+1}Q|}\int_{2^{j+1}Q}|f(y)|\,dy\\
&\leq C\cdot w(Q)^{{(1-\kappa)}/p}\cdot
\sum_{j=1}^\infty\frac{(j+1)\|b\|_*}{|2^{j+1}Q|}
\left(\int_{2^{j+1}Q}|f(y)|^p\nu(y)\,dy\right)^{1/p}\\
&\times\left(\int_{2^{j+1}Q}\nu(y)^{-p'/p}\,dy\right)^{1/{p'}}\\
&\leq C\big\|f\big\|_{L^{p,\kappa}(\nu,w)}\cdot w(Q)^{{(1-\kappa)}/p}\sum_{j=1}^\infty(j+1)\cdot\frac{w(2^{j+1}Q)^{\kappa/p}}{|2^{j+1}Q|}\\
&\times\left(\int_{2^{j+1}Q}\nu(y)^{-p'/p}\,dy\right)^{1/{p'}}\\
\end{split}
\end{equation*}
\begin{equation*}
\begin{split}
&= C\big\|f\big\|_{L^{p,\kappa}(\nu,w)}\sum_{j=1}^\infty(j+1)\cdot
\frac{w(Q)^{{(1-\kappa)}/p}}{w(2^{j+1}Q)^{{(1-\kappa)}/p}}\cdot\frac{w(2^{j+1}Q)^{1/p}}{|2^{j+1}Q|}\\
&\times\left(\int_{2^{j+1}Q}\nu(y)^{-p'/p}\,dy\right)^{1/{p'}}.
\end{split}
\end{equation*}
Let $\mathcal C(t)$ and $\mathcal A(t)$ be the same as before. Obviously, $\mathcal C(t)\leq\mathcal A(t)$ for all $t>0$, then for any cube $Q$ in $\mathbb R^n$, one has $\big\|f\big\|_{\mathcal C,Q}\leq\big\|f\big\|_{\mathcal A,Q}$ by definition, which implies that the condition $(\S\S)$ is stronger than the condition $(\S)$. This fact together with \eqref{U} yields
\begin{equation*}
\begin{split}
J_6&\leq C\big\|f\big\|_{L^{p,\kappa}(\nu,w)}\sum_{j=1}^\infty(j+1)\cdot
\frac{w(Q)^{{(1-\kappa)}/p}}{w(2^{j+1}Q)^{{(1-\kappa)}/p}}\cdot\frac{|2^{j+1}Q|^{1/{(r'p)}}}{|2^{j+1}Q|}\\
&\times\left(\int_{2^{j+1}Q}w(y)^r\,dy\right)^{1/{(rp)}}\left(\int_{2^{j+1}Q}\nu(y)^{-p'/p}\,dy\right)^{1/{p'}}\\
&\leq C\big\|f\big\|_{L^{p,\kappa}(\nu,w)}
\sum_{j=1}^\infty(j+1)\cdot\frac{w(Q)^{{(1-\kappa)}/p}}{w(2^{j+1}Q)^{{(1-\kappa)}/p}}.
\end{split}
\end{equation*}
Moreover, by our hypothesis on $w:w\in A_\infty$ and inequality (\ref{compare}), we compute
\begin{align}\label{C2}
\sum_{j=1}^\infty(j+1)\cdot\frac{w(Q)^{{(1-\kappa)}/p}}{w(2^{j+1}Q)^{{(1-\kappa)}/p}}
&\leq C\sum_{j=1}^\infty(j+1)\cdot\left(\frac{|Q|}{|2^{j+1}Q|}\right)^{{\delta(1-\kappa)}/p}\notag\\
&\leq C\sum_{j=1}^\infty(j+1)\cdot\left(\frac{1}{2^{(j+1)n}}\right)^{{\delta(1-\kappa)}/p}\leq C,
\end{align}
where the last series is convergent since the exponent $\delta {(1-\kappa)}/p$ is positive. This implies our desired estimate
\begin{equation*}
J_6\leq C\big\|f\big\|_{L^{p,\kappa}(\nu,w)}.
\end{equation*}
Summing up all the above estimates, and then taking the supremum over all cubes $Q\subset\mathbb R^n$ and all $\sigma>0$, we conclude the proof of Theorem \ref{mainthm:3}.
\end{proof}

\begin{proof}[Proof of Theorem $\ref{mainthm:4}$]
Let $f\in L^{p,\kappa}(\nu,w)$ with $1<p<\infty$ and $0<\kappa<1$. For any given cube $Q=Q(x_0,\ell)\subset\mathbb R^n$, as before, we set
\begin{equation*}
f=f_1+f_2,\qquad f_1=f\cdot\chi_{2Q},\quad  f_2=f\cdot\chi_{(2Q)^c},
\end{equation*}
where $2Q:=Q(x_0,2\sqrt{n}\ell)$. Then for any given $\sigma>0$, one writes
\begin{equation*}
\begin{split}
&\frac{1}{w(Q)^{\kappa/p}}\sigma\cdot \Big[w\big(\big\{x\in Q:\big|[b,\mathcal T_\gamma](f)(x)\big|>\sigma\big\}\big)\Big]^{1/p}\\
\leq &\frac{1}{w(Q)^{\kappa/p}}\sigma\cdot \Big[w\big(\big\{x\in Q:\big|[b,\mathcal T_\gamma](f_1)(x)\big|>\sigma/2\big\}\big)\Big]^{1/p}\\
&+\frac{1}{w(Q)^{\kappa/p}}\sigma\cdot \Big[w\big(\big\{x\in Q:\big|[b,\mathcal T_\gamma](f_2)(x)\big|>\sigma/2\big\}\big)\Big]^{1/p}\\
:=&J'_1+J'_2.
\end{split}
\end{equation*}
Since $w$ is an $A_{\infty}$ weight, then we have $w\in\Delta_2$. From our assumption \eqref{assump4.2} and inequality \eqref{weights}, it follows that
\begin{equation*}
\begin{split}
J'_1&\leq C\cdot\frac{1}{w(Q)^{\kappa/p}}\left(\int_{\mathbb R^n}|f_1(x)|^p \nu(x)\,dx\right)^{1/p}\\
&=C\cdot\frac{1}{w(Q)^{\kappa/p}}\left(\int_{2Q}|f(x)|^p \nu(x)\,dx\right)^{1/p}\\
&\leq C\big\|f\big\|_{L^{p,\kappa}(\nu,w)}\cdot\frac{w(2Q)^{\kappa/p}}{w(Q)^{\kappa/p}}\\
&\leq C\big\|f\big\|_{L^{p,\kappa}(\nu,w)}.
\end{split}
\end{equation*}
On the other hand, for any $x\in Q$, from the size condition (\ref{frac sublinear commutator}), it then follows that
\begin{equation*}
\begin{split}
\big|[b,\mathcal T_\gamma](f_2)(x)\big|&\leq c_4\int_{\mathbb R^n}\frac{|b(x)-b(y)|\cdot|f_2(y)|}{|x-y|^{n-\gamma}}dy\\
&\leq c_4\big|b(x)-b_{Q}\big|\cdot\int_{\mathbb R^n}\frac{|f_2(y)|}{|x-y|^{n-\gamma}}dy
+c_4\int_{\mathbb R^n}\frac{|b(y)-b_{Q}|\cdot|f_2(y)|}{|x-y|^{n-\gamma}}dy\\
&:=\widetilde\xi(x)+\widetilde\eta(x).
\end{split}
\end{equation*}
Thus, we can further split $J'_2$ into two parts as follows:
\begin{equation*}
\begin{split}
J'_2\leq&\frac{1}{w(Q)^{\kappa/p}}\sigma\cdot\Big[w\big(\big\{x\in Q:\widetilde\xi(x)>\sigma/4\big\}\big)\Big]^{1/p}
+\frac{1}{w(Q)^{\kappa/p}}\sigma\cdot\Big[w\big(\big\{x\in Q:\widetilde\eta(x)>\sigma/4\big\}\big)\Big]^{1/p}\\
:=&J'_3+J'_4.
\end{split}
\end{equation*}
Using the pointwise estimate \eqref{pointwise2} and Chebyshev's inequality, we obtain that
\begin{equation*}
\begin{split}
J'_3&\leq\frac{4}{w(Q)^{\kappa/p}}\cdot\left(\int_Q\big|\widetilde\xi(x)\big|^pw(x)\,dx\right)^{1/p}\\
&\leq\frac{C}{w(Q)^{\kappa/p}}\cdot\left(\int_Q \big|b(x)-b_{Q}\big|^pw(x)\,dx\right)^{1/p}
\sum_{j=1}^\infty\frac{1}{|2^{j+1}Q|^{1-\gamma/n}}\int_{2^{j+1}Q}|f(y)|\,dy\\
&\leq C\|b\|_*\cdot w(Q)^{{(1-\kappa)}/p}
\sum_{j=1}^\infty\frac{1}{|2^{j+1}Q|^{1-\gamma/n}}\int_{2^{j+1}Q}|f(y)|\,dy,
\end{split}
\end{equation*}
where the last inequality is due to $w\in A_{\infty}$ and Lemma \ref{BMO} $(ii)$. By using the same arguments as that of Theorem \ref{mainthm:2}, we can also show that
\begin{equation*}
J'_3\leq C\big\|f\big\|_{L^{p,\kappa}(\nu,w)}.
\end{equation*}
Similar to the proof of \eqref{pointwise2}, for all $x\in Q$, we can show the following pointwise inequality as well.
\begin{align}\label{pointwise4}
\big|\widetilde\eta(x)\big|
&= c_4\int_{\mathbb R^n}\frac{|b(y)-b_{Q}|\cdot|f_2(y)|}{|x-y|^{n-\gamma}}dy\notag\\
&\leq C\int_{(2Q)^c}\frac{|b(y)-b_{Q}|\cdot|f(y)|}{|x_0-y|^{n-\gamma}}dy\notag\\
&\leq C\sum_{j=1}^\infty\frac{1}{|2^{j+1}Q|^{1-\gamma/n}}\int_{2^{j+1}Q}\big|b(y)-b_{Q}\big|\cdot\big|f(y)\big|\,dy.
\end{align}
This, together with Chebyshev's inequality, yields
\begin{equation*}
\begin{split}
J'_4&\leq\frac{4}{w(Q)^{\kappa/p}}\cdot\left(\int_Q\big|\widetilde\eta(x)\big|^pw(x)\,dx\right)^{1/p}\\
&\leq C\cdot w(Q)^{{(1-\kappa)}/p}\cdot
\sum_{j=1}^\infty\frac{1}{|2^{j+1}Q|^{1-\gamma/n}}\int_{2^{j+1}Q}\big|b(y)-b_{Q}\big|\cdot\big|f(y)\big|\,dy\\
&\leq C\cdot w(Q)^{{(1-\kappa)}/p}\cdot
\sum_{j=1}^\infty\frac{1}{|2^{j+1}Q|^{1-\gamma/n}}\int_{2^{j+1}Q}\big|b(y)-b_{{2^{j+1}Q}}\big|\cdot\big|f(y)\big|\,dy\\
&+C\cdot w(Q)^{{(1-\kappa)}/p}\cdot
\sum_{j=1}^\infty\frac{1}{|2^{j+1}Q|^{1-\gamma/n}}\int_{2^{j+1}Q}\big|b_{{2^{j+1}Q}}-b_{Q}\big|\cdot\big|f(y)\big|\,dy\\
&:=J'_5+J'_6.
\end{split}
\end{equation*}
An application of H\"older's inequality leads to that
\begin{equation*}
\begin{split}
J'_5&\leq C\cdot w(Q)^{{(1-\kappa)}/p}\cdot\sum_{j=1}^\infty\frac{1}{|2^{j+1}Q|^{1-\gamma/n}}
\left(\int_{2^{j+1}Q}|f(y)|^p\nu(y)\,dy\right)^{1/p}\\
&\times\left(\int_{2^{j+1}Q}\big|b(y)-b_{{2^{j+1}Q}}\big|^{p'}\nu(y)^{-p'/p}\,dy\right)^{1/{p'}}\\
&\leq C\big\|f\big\|_{L^{p,\kappa}(\nu,w)}\cdot w(Q)^{{(1-\kappa)}/p}\sum_{j=1}^\infty\frac{w(2^{j+1}Q)^{\kappa/p}}{|2^{j+1}Q|^{1-\gamma/n}}\\
&\times\big|2^{j+1}Q\big|^{1/{p'}}\Big\|\big[b-b_{{2^{j+1}Q}}\big]\cdot \nu^{-1/p}\Big\|_{\mathcal C,2^{j+1}Q},
\end{split}
\end{equation*}
where $\mathcal C(t)=t^{p'}$ is a Young function. Let $\mathcal B(t)$ and $\mathcal A(t)$ be the same as in Theorem \ref{mainthm:3}. In view of \eqref{key} and \eqref{U}, we can deduce that
\begin{equation*}
\begin{split}
J'_5&\leq C\big\|f\big\|_{L^{p,\kappa}(\nu,w)}
\sum_{j=1}^\infty\frac{w(Q)^{{(1-\kappa)}/p}}{w(2^{j+1}Q)^{{(1-\kappa)}/p}}
\cdot\frac{w(2^{j+1}Q)^{1/p}}{|2^{j+1}Q|^{1/p-\gamma/n}}\cdot\|b\|_*\big\|\nu^{-1/p}\big\|_{\mathcal A,2^{j+1}Q}\\
&\leq C\|b\|_*\big\|f\big\|_{L^{p,\kappa}(\nu,w)}\sum_{j=1}^\infty\frac{w(Q)^{{(1-\kappa)}/p}}{w(2^{j+1}Q)^{{(1-\kappa)}/p}}\\
&\times\big|2^{j+1}Q\big|^{\gamma/n}\cdot
\left(\frac{1}{|2^{j+1}Q|}\int_{2^{j+1}Q}w(x)^r\,dx\right)^{1/{(rp)}}\cdot\big\|\nu^{-1/p}\big\|_{\mathcal A,2^{j+1}Q}.\\
\end{split}
\end{equation*}
Furthermore, by the $A_p$-type condition $(\S\S')$ on $(w,\nu)$ and the estimate \eqref{C1}, we obtain
\begin{equation*}
\begin{split}
J'_5&\leq C\big\|f\big\|_{L^{p,\kappa}(\nu,w)}\times\sum_{j=1}^\infty\frac{w(Q)^{{(1-\kappa)}/p}}{w(2^{j+1}Q)^{{(1-\kappa)}/p}}\\
&\leq C\big\|f\big\|_{L^{p,\kappa}(\nu,w)}.
\end{split}
\end{equation*}
It remains to estimate the last term $J'_6$. Making use of the first part of Lemma \ref{BMO} and H\"older's inequality, we get
\begin{equation*}
\begin{split}
J'_6&\leq C\cdot w(Q)^{{(1-\kappa)}/p}\cdot
\sum_{j=1}^\infty\frac{(j+1)\|b\|_*}{|2^{j+1}Q|^{1-\gamma/n}}\int_{2^{j+1}Q}|f(y)|\,dy\\
&\leq C\cdot w(Q)^{{(1-\kappa)}/p}\cdot
\sum_{j=1}^\infty\frac{(j+1)\|b\|_*}{|2^{j+1}Q|^{1-\gamma/n}}
\left(\int_{2^{j+1}Q}|f(y)|^p\nu(y)\,dy\right)^{1/p}\\
&\times\left(\int_{2^{j+1}Q}\nu(y)^{-p'/p}\,dy\right)^{1/{p'}}\\
&\leq C\big\|f\big\|_{L^{p,\kappa}(\nu,w)}\cdot w(Q)^{{(1-\kappa)}/p}
\sum_{j=1}^\infty(j+1)\cdot\frac{w(2^{j+1}Q)^{\kappa/p}}{|2^{j+1}Q|^{1-\gamma/n}}\\
&\times\left(\int_{2^{j+1}Q}\nu(y)^{-p'/p}\,dy\right)^{1/{p'}}\\
&=C\big\|f\big\|_{L^{p,\kappa}(\nu,w)}\sum_{j=1}^\infty(j+1)\cdot
\frac{w(Q)^{{(1-\kappa)}/p}}{w(2^{j+1}Q)^{{(1-\kappa)}/p}}\cdot\frac{w(2^{j+1}Q)^{1/p}}{|2^{j+1}Q|^{1-\gamma/n}}\\
&\times\left(\int_{2^{j+1}Q}\nu(y)^{-p'/p}\,dy\right)^{1/{p'}}.
\end{split}
\end{equation*}
It was pointed out in Theorem \ref{mainthm:3} that for any cube $Q$ in $\mathbb R^n$, one has $\big\|f\big\|_{\mathcal C,Q}\leq\big\|f\big\|_{\mathcal A,Q}$, where $\mathcal C(t)=t^{p'}$ and $\mathcal A(t)\approx t^{p'}\big[\log(e+t)\big]^{p'}$. This implies that the condition $(\S\S')$ is stronger than the condition $(\S')$. Using this fact along with \eqref{U}, we can see that
\begin{equation*}
\begin{split}
J'_6&\leq C\big\|f\big\|_{L^{p,\kappa}(\nu,w)}\sum_{j=1}^\infty(j+1)\cdot
\frac{w(Q)^{{(1-\kappa)}/p}}{w(2^{j+1}Q)^{{(1-\kappa)}/p}}\cdot\frac{|2^{j+1}Q|^{1/{(r'p)}}}{|2^{j+1}Q|^{1-\gamma/n}}\\
&\times\left(\int_{2^{j+1}Q}w(y)^r\,dy\right)^{1/{(rp)}}\left(\int_{2^{j+1}Q}\nu(y)^{-p'/p}\,dy\right)^{1/{p'}}\\
&\leq C\big\|f\big\|_{L^{p,\kappa}(\nu,w)}
\sum_{j=1}^\infty(j+1)\cdot\frac{w(Q)^{{(1-\kappa)}/p}}{w(2^{j+1}Q)^{{(1-\kappa)}/p}}\\
&\leq C\big\|f\big\|_{L^{p,\kappa}(\nu,w)},
\end{split}
\end{equation*}
where the last inequality follows from the estimate \eqref{C2}.Summarizing the estimates derived above, and then taking the supremum over all cubes $Q\subset\mathbb R^n$ and all $\sigma>0$, we therefore conclude the proof of Theorem \ref{mainthm:4}.
\end{proof}

\section{Proofs of Theorems \ref{mainthm:5} and \ref{mainthm:6}}

\begin{proof}[Proof of Theorem $\ref{mainthm:5}$]
Let $1<p\leq\alpha<q\leq\infty$ and $f\in(L^p,L^q)^{\alpha}(\nu,w;\mu)$ with $w\in\Delta_2$ and $\mu\in\Delta_2$. For any cube $Q=Q(y,\ell)\subset\mathbb R^n$ with $y\in\mathbb R^n$ and $\ell>0$, we denote by $\lambda Q$ the cube concentric with $Q$ whose each edge is $\lambda\sqrt{n}$ times as long, that is, $\lambda Q=Q(y,\lambda\sqrt{n}\ell)$. Decompose $f$ as
\begin{equation*}
\begin{cases}
f=f_1+f_2\in (L^p,L^q)^{\alpha}(\nu,w;\mu);\  &\\
f_1=f\cdot\chi_{2Q};\  &\\
f_2=f\cdot\chi_{(2Q)^c},
\end{cases}
\end{equation*}
where $\chi_{2Q}$ denotes the characteristic function of $2Q=Q(y,2\sqrt{n}\ell)$. Then for given $y\in\mathbb R^n$ and $\ell>0$, we write
\begin{align}\label{K}
&w(Q(y,\ell))^{1/{\alpha}-1/p-1/q}\big\|\mathcal T(f)\cdot\chi_{Q(y,\ell)}\big\|_{WL^p(w)}\notag\\
&\leq 2\cdot w(Q(y,\ell))^{1/{\alpha}-1/p-1/q}\big\|\mathcal T(f_1)\cdot\chi_{Q(y,\ell)}\big\|_{WL^p(w)}\notag\\
&+2\cdot w(Q(y,\ell))^{1/{\alpha}-1/p-1/q}\big\|\mathcal T(f_2)\cdot\chi_{Q(y,\ell)}\big\|_{WL^p(w)}\notag\\
&:=K_1(y,\ell)+K_2(y,\ell).
\end{align}
Let us consider the first term $K_1(y,\ell)$. In view of \eqref{assump1.2}, we get
\begin{align}\label{K1}
K_1(y,\ell)&\leq 2\cdot w(Q(y,\ell))^{1/{\alpha}-1/p-1/q}\big\|\mathcal T(f_1)\big\|_{WL^p(w)}\notag\\
&\leq C\cdot w(Q(y,\ell))^{1/{\alpha}-1/p-1/q}
\bigg(\int_{Q(y,2\sqrt{n}\ell)}|f(x)|^p\nu(x)\,dx\bigg)^{1/p}\notag\\
&=C\cdot w(Q(y,2\sqrt{n}\ell))^{1/{\alpha}-1/p-1/q}\big\|f\cdot\chi_{Q(y,2\sqrt{n}\ell)}\big\|_{L^p(\nu)}\notag\\
&\times \frac{w(Q(y,\ell))^{1/{\alpha}-1/p-1/q}}{w(Q(y,2\sqrt{n}\ell))^{1/{\alpha}-1/p-1/q}}.
\end{align}
Moreover, since $1/{\alpha}-1/p-1/q<0$ and $w\in \Delta_2$, then by doubling inequality \eqref{weights}, we obtain
\begin{equation}\label{doubling3}
\frac{w(Q(y,\ell))^{1/{\alpha}-1/p-1/q}}{w(Q(y,2\sqrt{n}\ell))^{1/{\alpha}-1/p-1/q}}\leq C.
\end{equation}
Substituting the above inequality \eqref{doubling3} into \eqref{K1}, we thus obtain
\begin{equation}\label{k1yr}
K_1(y,\ell)\leq C\cdot w(Q(y,2\sqrt{n}\ell))^{1/{\alpha}-1/p-1/q}\big\|f\cdot\chi_{Q(y,2\sqrt{n}\ell)}\big\|_{L^p(\nu)}.
\end{equation}
As for the second term $K_2(y,\ell)$, recall that by the size condition \eqref{sublinear}, the following inequality holds for any $x\in Q(y,\ell)$,
\begin{equation}\label{pointwise5}
\big|\mathcal T(f_2)(x)\big|\leq C
\sum_{j=1}^\infty\frac{1}{|Q(y,2^{j+1}\sqrt{n}\ell)|}\int_{Q(y,2^{j+1}\sqrt{n}\ell)}|f(z)|\,dz.
\end{equation}
This pointwise estimate \eqref{pointwise5} together with Chebyshev's inequality yields
\begin{equation*}
\begin{split}
K_2(y,\ell)&\leq 2\cdot w(Q(y,\ell))^{1/{\alpha}-1/p-1/q}\bigg(\int_{Q(y,\ell)}\big|\mathcal T(f_2)(x)\big|^pw(x)\,dx\bigg)^{1/p}\\
&\leq C\cdot w(Q(y,\ell))^{1/{\alpha}-1/q}
\sum_{j=1}^\infty\frac{1}{|Q(y,2^{j+1}\sqrt{n}\ell)|}\int_{Q(y,2^{j+1}\sqrt{n}\ell)}|f(z)|\,dz.
\end{split}
\end{equation*}
Moreover, an application of H\"older's inequality gives us that
\begin{equation*}
\begin{split}
K_2(y,\ell)&\leq C\cdot w(Q(y,\ell))^{1/{\alpha}-1/q}
\sum_{j=1}^\infty\frac{1}{|Q(y,2^{j+1}\sqrt{n}\ell)|}\bigg(\int_{Q(y,2^{j+1}\sqrt{n}\ell)}|f(z)|^p\nu(z)\,dz\bigg)^{1/p}\\
&\times\bigg(\int_{Q(y,2^{j+1}\sqrt{n}\ell)}\nu(z)^{-p'/p}\,dz\bigg)^{1/{p'}}\\
&=C\sum_{j=1}^\infty w\big(Q(y,2^{j+1}\sqrt{n}\ell)\big)^{1/{\alpha}-1/p-1/q}\big\|f\cdot\chi_{Q(y,2^{j+1}\sqrt{n}\ell)}\big\|_{L^p(\nu)}\\
&\times\frac{w(Q(y,\ell))^{1/{\alpha}-1/q}}{w(Q(y,2^{j+1}\sqrt{n}\ell))^{1/{\alpha}-1/q}}
\cdot\frac{w(Q(y,2^{j+1}\sqrt{n}\ell))^{1/p}}{|Q(y,2^{j+1}\sqrt{n}\ell)|}\bigg(\int_{Q(y,2^{j+1}\sqrt{n}\ell)}\nu(z)^{-p'/p}\,dz\bigg)^{1/{p'}}.
\end{split}
\end{equation*}
In addition, we apply H\"older's inequality with exponent $r>1$ to get
\begin{align}\label{U2}
w\big(Q(y,2^{j+1}\sqrt{n}\ell)\big)&=\int_{Q(y,2^{j+1}\sqrt{n}\ell)}w(z)\,dz\notag\\
&\leq\big|Q(y,2^{j+1}\sqrt{n}\ell)\big|^{1/{r'}}\bigg(\int_{Q(y,2^{j+1}\sqrt{n}\ell)}w(z)^r\,dz\bigg)^{1/r}.
\end{align}
Consequently,
\begin{equation}\label{k2yr}
\begin{split}
K_2(y,\ell)&\leq C\sum_{j=1}^\infty w\big(Q(y,2^{j+1}\sqrt{n}\ell)\big)^{1/{\alpha}-1/p-1/q}\big\|f\cdot\chi_{Q(y,2^{j+1}\sqrt{n}\ell)}\big\|_{L^p(\nu)}
\cdot\frac{w(Q(y,\ell))^{1/{\alpha}-1/q}}{w(Q(y,2^{j+1}\sqrt{n}\ell))^{1/{\alpha}-1/q}}\\
&\times\frac{|Q(y,2^{j+1}\sqrt{n}\ell)|^{1/{(r'p)}}}{|Q(y,2^{j+1}\sqrt{n}\ell)|}
\bigg(\int_{Q(y,2^{j+1}\sqrt{n}\ell)}w(z)^r\,dz\bigg)^{1/{(rp)}}
\bigg(\int_{Q(y,2^{j+1}\sqrt{n}\ell)}\nu(z)^{-p'/p}\,dz\bigg)^{1/{p'}}\\
&\leq C\sum_{j=1}^\infty w\big(Q(y,2^{j+1}\sqrt{n}\ell)\big)^{1/{\alpha}-1/p-1/q}\big\|f\cdot\chi_{Q(y,2^{j+1}\sqrt{n}\ell)}\big\|_{L^p(\nu)}
\cdot\frac{w(Q(y,\ell))^{1/{\alpha}-1/q}}{w(Q(y,2^{j+1}\sqrt{n}\ell))^{1/{\alpha}-1/q}}.
\end{split}
\end{equation}
The last inequality is obtained by the $A_p$-type condition $(\S)$ on $(w,\nu)$. Furthermore, arguing as in the proof of Theorem \ref{mainthm:1}, we know that for any positive integer $j$, there exists a \emph{reverse doubling constant }$D=D(w)>1$ independent of $Q(y,\ell)$ such that
\begin{equation*}
w\big(Q(y,2^{j+1}\sqrt{n}\ell)\big)\geq D^{j+1}\cdot w(Q(y,\ell)).
\end{equation*}
Hence,
\begin{align}\label{5}
\sum_{j=1}^\infty\frac{w(Q(y,\ell))^{1/{\alpha}-1/q}}{w(Q(y,2^{j+1}\sqrt{n}\ell))^{1/{\alpha}-1/q}}
&\leq \sum_{j=1}^\infty\left(\frac{w(Q(y,\ell))}{D^{j+1}\cdot w(Q(y,\ell))}\right)^{1/{\alpha}-1/q}\notag\\
&=\sum_{j=1}^\infty\left(\frac{1}{D^{j+1}}\right)^{1/{\alpha}-1/q}\notag\\
&\leq C,
\end{align}
where the last series is convergent since the \emph{reverse doubling constant }$D>1$ and $1/{\alpha}-1/q>0$.
Therefore by taking the $L^q(\mu)$-norm of both sides of \eqref{K}(with respect to the variable $y$), and then using Minkowski's inequality, \eqref{k1yr} and \eqref{k2yr}, we have
\begin{equation*}
\begin{split}
&\Big\|w(Q(y,\ell))^{1/{\alpha}-1/p-1/q}\big\|\mathcal T(f)\cdot\chi_{Q(y,\ell)}\big\|_{WL^p(w)}\Big\|_{L^q(\mu)}\\
&\leq\big\|K_1(y,\ell)\big\|_{L^q(\mu)}+\big\|K_2(y,\ell)\big\|_{L^q(\mu)}\\
&\leq C\Big\|w(Q(y,2\sqrt{n}\ell))^{1/{\alpha}-1/p-1/q}\big\|f\cdot\chi_{Q(y,2\sqrt{n}\ell)}\big\|_{L^p(\nu)}\Big\|_{L^q(\mu)}\\
&+C\sum_{j=1}^\infty\Big\|w\big(Q(y,2^{j+1}\sqrt{n}\ell)\big)^{1/{\alpha}-1/p-1/q}\big\|f\cdot\chi_{Q(y,2^{j+1}\sqrt{n}\ell)}\big\|_{L^p(\nu)}\Big\|_{L^q(\mu)}
\times\frac{w(Q(y,\ell))^{1/{\alpha}-1/q}}{w(Q(y,2^{j+1}\sqrt{n}\ell))^{1/{\alpha}-1/q}}\\
&\leq C\big\|f\big\|_{(L^p,L^q)^{\alpha}(\nu,w;\mu)}+C\big\|f\big\|_{(L^p,L^q)^{\alpha}(\nu,w;\mu)}
\times\sum_{j=1}^\infty\frac{w(Q(y,\ell))^{1/{\alpha}-1/q}}{w(Q(y,2^{j+1}\sqrt{n}\ell))^{1/{\alpha}-1/q}}\\
&\leq C\big\|f\big\|_{(L^p,L^q)^{\alpha}(\nu,w;\mu)},
\end{split}
\end{equation*}
where the last inequality follows from \eqref{5}. By taking the supremum over all $\ell>0$, we finish the proof of Theorem \ref{mainthm:5}.
\end{proof}

\begin{proof}[Proof of Theorem $\ref{mainthm:6}$]
Let $1<p\leq\alpha<q\leq\infty$ and $f\in(L^p,L^q)^{\alpha}(\nu,w;\mu)$ with $w\in\Delta_2$ and $\mu\in\Delta_2$. For an arbitrary point $y\in\mathbb R^n$, we set $Q=Q(y,\ell)$ for the cube centered at $y$ and of the side length $\ell$.
Decompose $f$ as
\begin{equation*}
\begin{cases}
f=f_1+f_2\in (L^p,L^q)^{\alpha}(\nu,w;\mu);\  &\\
f_1=f\cdot\chi_{2Q};\  &\\
f_2=f\cdot\chi_{(2Q)^c},
\end{cases}
\end{equation*}
where $2Q=Q(y,2\sqrt{n}\ell)$. Then for given $y\in\mathbb R^n$ and $\ell>0$, we write
\begin{align}\label{Kp}
&w(Q(y,\ell))^{1/{\alpha}-1/p-1/q}\big\|\mathcal T_{\gamma}(f)\cdot\chi_{Q(y,\ell)}\big\|_{WL^p(w)}\notag\\
&\leq 2\cdot w(Q(y,\ell))^{1/{\alpha}-1/p-1/q}\big\|\mathcal T_{\gamma}(f_1)\cdot\chi_{Q(y,\ell)}\big\|_{WL^p(w)}\notag\\
&+2\cdot w(Q(y,\ell))^{1/{\alpha}-1/p-1/q}\big\|\mathcal T_{\gamma}(f_2)\cdot\chi_{Q(y,\ell)}\big\|_{WL^p(w)}\notag\\
&:=K'_1(y,\ell)+K'_2(y,\ell).
\end{align}
Let us consider the first term $K'_1(y,\ell)$. Using the assumption \eqref{assump2.2} and inequality \eqref{doubling3}, we get
\begin{align}\label{kp1yr}
K'_1(y,\ell)&\leq 2\cdot w(Q(y,\ell))^{1/{\alpha}-1/p-1/q}\big\|\mathcal T_\gamma(f_1)\big\|_{WL^p(w)}\notag\\
&\leq C\cdot w(Q(y,\ell))^{1/{\alpha}-1/p-1/q}
\bigg(\int_{Q(y,2\sqrt{n}\ell)}|f(x)|^p\nu(x)\,dx\bigg)^{1/p}\notag\\
&=C\cdot w(Q(y,2\sqrt{n}\ell))^{1/{\alpha}-1/p-1/q}\big\|f\cdot\chi_{Q(y,2\sqrt{n}\ell)}\big\|_{L^p(\nu)}\notag\\
&\times \frac{w(Q(y,\ell))^{1/{\alpha}-1/p-1/q}}{w(Q(y,2\sqrt{n}\ell))^{1/{\alpha}-1/p-1/q}}\notag\\
&\leq C\cdot w(Q(y,2\sqrt{n}\ell))^{1/{\alpha}-1/p-1/q}\big\|f\cdot\chi_{Q(y,2\sqrt{n}\ell)}\big\|_{L^p(\nu)}.
\end{align}
We now estimate the second term $K'_2(y,\ell)$. Recall that by the size condition \eqref{frac sublinear}, the following estimate holds for any $x\in Q(y,\ell)$,
\begin{equation}\label{alpha1}
\big|\mathcal T_{\gamma}(f_2)(x)\big|\leq C\sum_{j=1}^\infty\frac{1}{|Q(y,2^{j+1}\sqrt{n}\ell)|^{1-{\gamma}/n}}\int_{Q(y,2^{j+1}\sqrt{n}\ell)}|f(z)|\,dz.
\end{equation}
This pointwise estimate \eqref{alpha1} along with Chebyshev's inequality implies
\begin{equation*}
\begin{split}
K'_2(y,\ell)&\leq 2\cdot w(Q(y,\ell))^{1/{\alpha}-1/p-1/q}\bigg(\int_{Q(y,\ell)}\big|\mathcal T_{\gamma}(f_2)(x)\big|^pw(x)\,dx\bigg)^{1/p}\\
&\leq C\cdot w(Q(y,\ell))^{1/{\alpha}-1/q}
\sum_{j=1}^\infty\frac{1}{|Q(y,2^{j+1}\sqrt{n}\ell)|^{1-{\gamma}/n}}\int_{Q(y,2^{j+1}\sqrt{n}\ell)}|f(z)|\,dz.
\end{split}
\end{equation*}
A further application of H\"older's inequality yields
\begin{equation*}
\begin{split}
K'_2(y,\ell)&\leq C\cdot w(Q(y,\ell))^{1/{\alpha}-1/q}
\sum_{j=1}^\infty\frac{1}{|Q(y,2^{j+1}\sqrt{n}\ell)|^{1-{\gamma}/n}}\bigg(\int_{Q(y,2^{j+1}\sqrt{n}\ell)}|f(z)|^p\nu(z)\,dz\bigg)^{1/p}\\
&\times\bigg(\int_{Q(y,2^{j+1}\sqrt{n}\ell)}\nu(z)^{-p'/p}\,dz\bigg)^{1/{p'}}\\
&=C\sum_{j=1}^\infty w\big(Q(y,2^{j+1}\sqrt{n}\ell)\big)^{1/{\alpha}-1/p-1/q}\big\|f\cdot\chi_{Q(y,2^{j+1}\sqrt{n}\ell)}\big\|_{L^p(\nu)}\\
&\times\frac{w(Q(y,\ell))^{1/{\alpha}-1/q}}{w(Q(y,2^{j+1}\sqrt{n}\ell))^{1/{\alpha}-1/q}}
\cdot\frac{w(Q(y,2^{j+1}\sqrt{n}\ell))^{1/p}}{|Q(y,2^{j+1}\sqrt{n}\ell)|^{1-{\gamma}/n}}\\
&\times\bigg(\int_{Q(y,2^{j+1}\sqrt{n}\ell)}\nu(z)^{-p'/p}\,dz\bigg)^{1/{p'}}.
\end{split}
\end{equation*}
Hence, in view of \eqref{U2}, we have
\begin{equation}\label{kp2yr}
\begin{split}
K'_2(y,\ell)&\leq C\sum_{j=1}^\infty w\big(Q(y,2^{j+1}\sqrt{n}\ell)\big)^{1/{\alpha}-1/p-1/q}\big\|f\cdot\chi_{Q(y,2^{j+1}\sqrt{n}\ell)}\big\|_{L^p(\nu)}
\cdot\frac{w(Q(y,\ell))^{1/{\alpha}-1/q}}{w(Q(y,2^{j+1}\sqrt{n}\ell))^{1/{\alpha}-1/q}}\\
&\times\frac{|Q(y,2^{j+1}\sqrt{n}\ell)|^{1/{(r'p)}}}{|Q(y,2^{j+1}\sqrt{n}\ell)|^{1-{\gamma}/n}}
\bigg(\int_{Q(y,2^{j+1}\sqrt{n}\ell)}w(z)^r\,dz\bigg)^{1/{(rp)}}
\bigg(\int_{Q(y,2^{j+1}\sqrt{n}\ell)}\nu(z)^{-p'/p}\,dz\bigg)^{1/{p'}}\\
&\leq C\sum_{j=1}^\infty w\big(Q(y,2^{j+1}\sqrt{n}\ell)\big)^{1/{\alpha}-1/p-1/q}\big\|f\cdot\chi_{Q(y,2^{j+1}\sqrt{n}\ell)}\big\|_{L^p(\nu)}
\cdot\frac{w(Q(y,\ell))^{1/{\alpha}-1/q}}{w(Q(y,2^{j+1}\sqrt{n}\ell))^{1/{\alpha}-1/q}}.
\end{split}
\end{equation}
The last inequality is obtained by the $A_p$-type condition $(\S')$ on $(w,\nu)$. Therefore by taking the $L^q(\mu)$-norm of both sides of \eqref{Kp}(with respect to the variable $y$), and then using Minkowski's inequality, \eqref{kp1yr} and \eqref{kp2yr}, we obtain
\begin{equation*}
\begin{split}
&\Big\|w(Q(y,\ell))^{1/{\alpha}-1/p-1/q}\big\|\mathcal T_{\gamma}(f)\cdot\chi_{Q(y,\ell)}\big\|_{WL^p(w)}\Big\|_{L^q(\mu)}\\
&\leq\big\|K'_1(y,\ell)\big\|_{L^q(\mu)}+\big\|K'_2(y,\ell)\big\|_{L^q(\mu)}\\
&\leq C\Big\|w(Q(y,2\sqrt{n}\ell))^{1/{\alpha}-1/p-1/q}\big\|f\cdot\chi_{Q(y,2\sqrt{n}\ell)}\big\|_{L^p(\nu)}\Big\|_{L^q(\mu)}\\
&+C\sum_{j=1}^\infty\Big\|w\big(Q(y,2^{j+1}\sqrt{n}\ell)\big)^{1/{\alpha}-1/p-1/q}\big\|f\cdot\chi_{Q(y,2^{j+1}\sqrt{n}\ell)}\big\|_{L^p(\nu)}\Big\|_{L^q(\mu)}
\times\frac{w(Q(y,\ell))^{1/{\alpha}-1/q}}{w(Q(y,2^{j+1}\sqrt{n}\ell))^{1/{\alpha}-1/q}}\\
&\leq C\big\|f\big\|_{(L^p,L^q)^{\alpha}(\nu,w;\mu)}+C\big\|f\big\|_{(L^p,L^q)^{\alpha}(\nu,w;\mu)}
\times\sum_{j=1}^\infty\frac{w(Q(y,\ell))^{1/{\alpha}-1/q}}{w(Q(y,2^{j+1}\sqrt{n}\ell))^{1/{\alpha}-1/q}}\\
&\leq C\big\|f\big\|_{(L^p,L^q)^{\alpha}(\nu,w;\mu)},
\end{split}
\end{equation*}
where the last inequality follows from \eqref{5}. Thus, by taking the supremum over all $\ell>0$, we complete the proof of Theorem \ref{mainthm:6}.
\end{proof}

\section{Proofs of Theorems \ref{mainthm:7} and \ref{mainthm:8}}

\begin{proof}[Proof of Theorem $\ref{mainthm:7}$]
Let $1<p\leq\alpha<q\leq\infty$ and $f\in(L^p,L^q)^{\alpha}(\nu,w;\mu)$ with $w\in A_\infty$ and $\mu\in\Delta_2$. For any fixed cube $Q=Q(y,\ell)$ in $\mathbb R^n$, as before, we decompose $f$ as
\begin{equation*}
f=f_1+f_2,\qquad f_1=f\cdot\chi_{2Q},\quad  f_2=f\cdot\chi_{(2Q)^c},
\end{equation*}
where $2Q=Q(y,2\sqrt{n}\ell)$. Then for given $y\in\mathbb R^n$ and $\ell>0$, we write
\begin{align}\label{L}
&w(Q(y,\ell))^{1/{\alpha}-1/p-1/q}\big\|[b,\mathcal T](f)\cdot\chi_{Q(y,\ell)}\big\|_{WL^p(w)}\notag\\
&\leq 2\cdot w(Q(y,\ell))^{1/{\alpha}-1/p-1/q}\big\|[b,\mathcal T](f_1)\cdot\chi_{Q(y,\ell)}\big\|_{WL^p(w)}\notag\\
&+2\cdot w(Q(y,\ell))^{1/{\alpha}-1/p-1/q}\big\|[b,\mathcal T](f_2)\cdot\chi_{Q(y,\ell)}\big\|_{WL^p(w)}\notag\\
&:=L_1(y,\ell)+L_2(y,\ell).
\end{align}
Since $w\in A_\infty$, we know that $w\in\Delta_2$. By the assumption \eqref{assump3.2} and inequality (\ref{doubling3}), then we have
\begin{align}\label{L1}
L_1(y,\ell)&\leq 2\cdot w(Q(y,\ell))^{1/{\alpha}-1/p-1/q}\big\|[b,\mathcal T](f_1)\big\|_{WL^p(w)}\notag\\
&\leq C\cdot w(Q(y,\ell))^{1/{\alpha}-1/p-1/q}
\bigg(\int_{Q(y,2\sqrt{n}\ell)}|f(x)|^p\nu(x)\,dx\bigg)^{1/p}\notag\\
&=C\cdot w(Q(y,2\sqrt{n}\ell))^{1/{\alpha}-1/p-1/q}\big\|f\cdot\chi_{Q(y,2\sqrt{n}\ell)}\big\|_{L^p(\nu)}\notag\\
&\times \frac{w(Q(y,\ell))^{1/{\alpha}-1/p-1/q}}{w(Q(y,2\sqrt{n}\ell))^{1/{\alpha}-1/p-1/q}}\notag\\
&\leq C\cdot w(Q(y,2\sqrt{n}\ell))^{1/{\alpha}-1/p-1/q}\big\|f\cdot\chi_{Q(y,2\sqrt{n}\ell)}\big\|_{L^p(\nu)}.
\end{align}
Next we estimate $L_2(y,\ell)$. For any $x\in Q(y,\ell)$, from the condition \eqref{sublinear commutator}, we can see that
\begin{equation*}
\begin{split}
\big|[b,\mathcal T](f_2)(x)\big|
&\leq \big|b(x)-b_{Q(y,\ell)}\big|\cdot\big|\mathcal T(f_2)(x)\big|
+\Big|\mathcal T\big([b_{Q(y,\ell)}-b]f_2\big)(x)\Big|\\
&:=\xi_\ast(x)+\eta_\ast(x).
\end{split}
\end{equation*}
Then we have
\begin{equation*}
\begin{split}
L_2(y,\ell)\leq&4\cdot w(Q(y,\ell))^{1/{\alpha}-1/p-1/q}\big\|\xi_\ast(\cdot)\cdot\chi_{Q(y,\ell)}\big\|_{WL^p(w)}\\
&+4\cdot w(Q(y,\ell))^{1/{\alpha}-1/p-1/q}\big\|\eta_\ast(\cdot)\cdot\chi_{Q(y,\ell)}\big\|_{WL^p(w)}\\
:=&L_3(y,\ell)+L_4(y,\ell).
\end{split}
\end{equation*}
For the term $L_3(y,\ell)$, it follows directly from Chebyshev's inequality and estimate \eqref{pointwise5} that
\begin{equation*}
\begin{split}
L_3(y,\ell)&\leq4\cdot w(Q(y,\ell))^{1/{\alpha}-1/p-1/q}\bigg(\int_{Q(y,\ell)}\big|\xi_\ast(x)\big|^pw(x)\,dx\bigg)^{1/p}\\
&\leq C\cdot w(Q(y,\ell))^{1/{\alpha}-1/p-1/q}\bigg(\int_{Q(y,\ell)}\big|b(x)-b_{Q(y,\ell)}\big|^pw(x)\,dx\bigg)^{1/p}\\
&\times\sum_{j=1}^\infty\frac{1}{|Q(y,2^{j+1}\sqrt{n}\ell)|}\int_{Q(y,2^{j+1}\sqrt{n}\ell)}|f(z)|\,dz\\
&\leq C\|b\|_*\cdot w(Q(y,\ell))^{1/{\alpha}-1/q}
\sum_{j=1}^\infty\frac{1}{|Q(y,2^{j+1}\sqrt{n}\ell)|}\int_{Q(y,2^{j+1}\sqrt{n}\ell)}|f(z)|\,dz,
\end{split}
\end{equation*}
where in the last inequality we have used the fact that $w\in A_{\infty}$ and Lemma \ref{BMO}$(ii)$. We can now argue exactly as we did in the proof of Theorem \ref{mainthm:5} to get
\begin{equation*}
\begin{split}
L_3(y,\ell)&\leq C\sum_{j=1}^\infty w\big(Q(y,2^{j+1}\sqrt{n}\ell)\big)^{1/{\alpha}-1/p-1/q}\big\|f\cdot\chi_{Q(y,2^{j+1}\sqrt{n}\ell)}\big\|_{L^p(\nu)}
\cdot\frac{w(Q(y,\ell))^{1/{\alpha}-1/q}}{w(Q(y,2^{j+1}\sqrt{n}\ell))^{1/{\alpha}-1/q}}.
\end{split}
\end{equation*}
For the term $L_4(y,\ell)$, as it was shown in Theorem \ref{mainthm:3}, the following pointwise inequality holds by the size condition \eqref{sublinear}.
\begin{equation*}
\begin{split}
\eta_\ast(x)&=\Big|\mathcal T\big([b_{Q(y,\ell)}-b]f_2\big)(x)\Big|\\
&\leq C\sum_{j=1}^\infty\frac{1}{|Q(y,2^{j+1}\sqrt{n}\ell)|}\int_{Q(y,2^{j+1}\sqrt{n}\ell)}\big|b(z)-b_{Q(y,\ell)}\big|\cdot|f(z)|\,dz.
\end{split}
\end{equation*}
This, together with Chebyshev's inequality, yields
\begin{equation*}
\begin{split}
L_4(y,\ell)&\leq4\cdot w(Q(y,\ell))^{1/{\alpha}-1/p-1/q}\bigg(\int_{Q(y,\ell)}\big|\eta_\ast(x)\big|^pw(x)\,dx\bigg)^{1/p}\\
&\leq C\cdot w(Q(y,\ell))^{1/{\alpha}-1/q}
\sum_{j=1}^\infty\frac{1}{|Q(y,2^{j+1}\sqrt{n}\ell)|}\int_{Q(y,2^{j+1}\sqrt{n}\ell)}\big|b(z)-b_{Q(y,\ell)}\big|\cdot|f(z)|\,dz\\
&\leq C\cdot w(Q(y,\ell))^{1/{\alpha}-1/q}
\sum_{j=1}^\infty\frac{1}{|Q(y,2^{j+1}\sqrt{n}\ell)|}\int_{Q(y,2^{j+1}\sqrt{n}\ell)}\big|b(z)-b_{Q(y,2^{j+1}\sqrt{n}\ell)}\big|\cdot|f(z)|\,dz\\
&+C\cdot w(Q(y,\ell))^{1/{\alpha}-1/q}
\sum_{j=1}^\infty\frac{1}{|Q(y,2^{j+1}\sqrt{n}\ell)|}\int_{Q(y,2^{j+1}\sqrt{n}\ell)}\big|b_{Q(y,2^{j+1}\sqrt{n}\ell)}-b_{Q(y,\ell)}\big|\cdot|f(z)|\,dz\\
&:=L_5(y,\ell)+L_6(y,\ell).
\end{split}
\end{equation*}
An application of H\"older's inequality leads to that
\begin{equation*}
\begin{split}
L_5(y,\ell)&\leq C\cdot w(Q(y,\ell))^{1/{\alpha}-1/q}
\sum_{j=1}^\infty\frac{1}{|Q(y,2^{j+1}\sqrt{n}\ell)|}\bigg(\int_{Q(y,2^{j+1}\sqrt{n}\ell)}|f(z)|^p\nu(z)\,dz\bigg)^{1/p}\\
&\times\bigg(\int_{Q(y,2^{j+1}\sqrt{n}\ell)}\big|b(z)-b_{Q(y,2^{j+1}\sqrt{n}\ell)}\big|^{p'}\nu(z)^{-p'/p}\,dz\bigg)^{1/{p'}}\\
&\leq C\cdot w(Q(y,\ell))^{1/{\alpha}-1/q}\sum_{j=1}^\infty\frac{\big\|f\cdot\chi_{Q(y,2^{j+1}\sqrt{n}\ell)}\big\|_{L^p(\nu)}}{|Q(y,2^{j+1}\sqrt{n}\ell)|}\\
&\times\big|Q(y,2^{j+1}\sqrt{n}\ell)\big|^{1/{p'}}\Big\|\big[b-b_{Q(y,2^{j+1}\sqrt{n}\ell)}\big]\cdot \nu^{-1/p}\Big\|_{\mathcal C,Q(y,2^{j+1}\sqrt{n}\ell)},
\end{split}
\end{equation*}
where $\mathcal C(t)=t^{p'}$ is a Young function. Recall that the following inequality holds by generalized H\"older's inequality and the estimate \eqref{Jensen}:
\begin{align}\label{final}
\Big\|\big[b-b_{Q(y,2^{j+1}\sqrt{n}\ell)}\big]\cdot\nu^{-1/p}\Big\|_{\mathcal C,Q(y,2^{j+1}\sqrt{n}\ell)}
&\leq C\Big\|b-b_{Q(y,2^{j+1}\sqrt{n}\ell)}\Big\|_{\mathcal B,Q(y,2^{j+1}\sqrt{n}\ell)}\cdot\Big\|\nu^{-1/p}\Big\|_{\mathcal A,Q(y,2^{j+1}\sqrt{n}\ell)}\notag\\
&\leq C\|b\|_*\cdot\Big\|\nu^{-1/p}\Big\|_{\mathcal A,Q(y,2^{j+1}\sqrt{n}\ell)}.
\end{align}
where
\begin{equation*}
\mathcal A(t)\approx t^{p'}\big[\log(e+t)\big]^{p'}\qquad \&\qquad \mathcal B(t)\approx \exp(t)-1.
\end{equation*}
Consequently,
\begin{equation*}
\begin{split}
L_5(y,\ell)&\leq C\|b\|_*\cdot w(Q(y,\ell))^{1/{\alpha}-1/q}
\sum_{j=1}^\infty\frac{\big\|f\cdot\chi_{Q(y,2^{j+1}\sqrt{n}\ell)}\big\|_{L^p(\nu)}}{|Q(y,2^{j+1}\sqrt{n}\ell)|^{1/p}}
\cdot\Big\|\nu^{-1/p}\Big\|_{\mathcal A,Q(y,2^{j+1}\sqrt{n}\ell)}\\
&=C\|b\|_*\sum_{j=1}^\infty w\big(Q(y,2^{j+1}\sqrt{n}\ell)\big)^{1/{\alpha}-1/p-1/q}\big\|f\cdot\chi_{Q(y,2^{j+1}\sqrt{n}\ell)}\big\|_{L^p(\nu)}
\cdot\frac{w(Q(y,\ell))^{1/{\alpha}-1/q}}{w(Q(y,2^{j+1}\sqrt{n}\ell))^{1/{\alpha}-1/q}}\\
&\times\frac{w(Q(y,2^{j+1}\sqrt{n}\ell))^{1/p}}{|Q(y,2^{j+1}\sqrt{n}\ell)|^{1/p}}
\cdot\Big\|\nu^{-1/p}\Big\|_{\mathcal A,Q(y,2^{j+1}\sqrt{n}\ell)}.
\end{split}
\end{equation*}
Moreover, in view of \eqref{U2}, we can deduce that
\begin{equation*}
\begin{split}
L_5(y,\ell)
&\leq C\|b\|_*\sum_{j=1}^\infty w\big(Q(y,2^{j+1}\sqrt{n}\ell)\big)^{1/{\alpha}-1/p-1/q}\big\|f\cdot\chi_{Q(y,2^{j+1}\sqrt{n}\ell)}\big\|_{L^p(\nu)}
\cdot\frac{w(Q(y,\ell))^{1/{\alpha}-1/q}}{w(Q(y,2^{j+1}\sqrt{n}\ell))^{1/{\alpha}-1/q}}\\
&\times\frac{|Q(y,2^{j+1}\sqrt{n}\ell)|^{1/{(r'p)}}}{|Q(y,2^{j+1}\sqrt{n}\ell)|^{1/p}}
\bigg(\int_{Q(y,2^{j+1}\sqrt{n}\ell)}w(z)^r\,dz\bigg)^{1/{(rp)}}\cdot\Big\|\nu^{-1/p}\Big\|_{\mathcal A,Q(y,2^{j+1}\sqrt{n}\ell)}\\
&\leq C\|b\|_*\sum_{j=1}^\infty w\big(Q(y,2^{j+1}\sqrt{n}\ell)\big)^{1/{\alpha}-1/p-1/q}\big\|f\cdot\chi_{Q(y,2^{j+1}\sqrt{n}\ell)}\big\|_{L^p(\nu)}
\cdot\frac{w(Q(y,\ell))^{1/{\alpha}-1/q}}{w(Q(y,2^{j+1}\sqrt{n}\ell))^{1/{\alpha}-1/q}}.
\end{split}
\end{equation*}
The last inequality is obtained by the $A_p$-type condition $(\S\S)$ on $(w,\nu)$. Let us now estimate the last term $L_6(y,\ell)$. Applying Lemma \ref{BMO}$(i)$ and H\"older's inequality, we get
\begin{equation*}
\begin{split}
L_6(y,\ell)&\leq C\cdot w(Q(y,\ell))^{1/{\alpha}-1/q}
\sum_{j=1}^\infty\frac{(j+1)\|b\|_*}{|Q(y,2^{j+1}\sqrt{n}\ell)|}\int_{Q(y,2^{j+1}\sqrt{n}\ell)}|f(z)|\,dz\\
&\leq C\cdot w(Q(y,\ell))^{1/{\alpha}-1/q}
\sum_{j=1}^\infty\frac{(j+1)\|b\|_*}{|Q(y,2^{j+1}\sqrt{n}\ell)|}\bigg(\int_{Q(y,2^{j+1}\sqrt{n}\ell)}|f(z)|^p\nu(z)\,dz\bigg)^{1/p}\\
&\times\bigg(\int_{Q(y,2^{j+1}\sqrt{n}\ell)}\nu(z)^{-p'/p}\,dz\bigg)^{1/{p'}}\\
&=C\|b\|_*\sum_{j=1}^\infty w\big(Q(y,2^{j+1}\sqrt{n}\ell)\big)^{1/{\alpha}-1/p-1/q}\big\|f\cdot\chi_{Q(y,2^{j+1}\sqrt{n}\ell)}\big\|_{L^p(\nu)}\\
&\times\big(j+1\big)\cdot\frac{w(Q(y,\ell))^{1/{\alpha}-1/q}}{w(Q(y,2^{j+1}\sqrt{n}\ell))^{1/{\alpha}-1/q}}
\cdot\frac{w(Q(y,2^{j+1}\sqrt{n}\ell))^{1/p}}{|Q(y,2^{j+1}\sqrt{n}\ell)|}\\
&\times\bigg(\int_{Q(y,2^{j+1}\sqrt{n}\ell)}\nu(z)^{-p'/p}\,dz\bigg)^{1/{p'}}.
\end{split}
\end{equation*}
As it was pointed out in Theorem \ref{mainthm:3} that the condition $(\S\S)$ is stronger than the condition $(\S)$. Taking into account this fact and \eqref{U2}, we obtain
\begin{equation*}
\begin{split}
L_6(y,\ell)&\leq C\|b\|_*\sum_{j=1}^\infty w\big(Q(y,2^{j+1}\sqrt{n}\ell)\big)^{1/{\alpha}-1/p-1/q}\big\|f\cdot\chi_{Q(y,2^{j+1}\sqrt{n}\ell)}\big\|_{L^p(\nu)}\\
&\times\big(j+1\big)\cdot\frac{w(Q(y,\ell))^{1/{\alpha}-1/q}}{w(Q(y,2^{j+1}\sqrt{n}\ell))^{1/{\alpha}-1/q}}\\
&\times\frac{|Q(y,2^{j+1}\sqrt{n}\ell)|^{1/{(r'p)}}}{|Q(y,2^{j+1}\sqrt{n}\ell)|}
\bigg(\int_{Q(y,2^{j+1}\sqrt{n}\ell)}w(z)^r\,dz\bigg)^{1/{(rp)}}
\bigg(\int_{Q(y,2^{j+1}\sqrt{n}\ell)}\nu(z)^{-p'/p}\,dz\bigg)^{1/{p'}}\\
&\leq C\|b\|_*\sum_{j=1}^\infty w\big(Q(y,2^{j+1}\sqrt{n}\ell)\big)^{1/{\alpha}-1/p-1/q}\big\|f\cdot\chi_{Q(y,2^{j+1}\sqrt{n}\ell)}\big\|_{L^p(\nu)}\\
&\times\big(j+1\big)\cdot\frac{w(Q(y,\ell))^{1/{\alpha}-1/q}}{w(Q(y,2^{j+1}\sqrt{n}\ell))^{1/{\alpha}-1/q}}.
\end{split}
\end{equation*}
Summing up all the above estimates, we get
\begin{align}\label{L2}
L_2(y,\ell)&\leq C\sum_{j=1}^\infty w\big(Q(y,2^{j+1}\sqrt{n}\ell)\big)^{1/{\alpha}-1/p-1/q}\big\|f\cdot\chi_{Q(y,2^{j+1}\sqrt{n}\ell)}\big\|_{L^p(\nu)}\notag\\
&\times\big(j+1\big)\cdot\frac{w(Q(y,\ell))^{1/{\alpha}-1/q}}{w(Q(y,2^{j+1}\sqrt{n}\ell))^{1/{\alpha}-1/q}}.
\end{align}
Moreover, by our hypothesis on $w:w\in A_\infty$ and inequality \eqref{compare} with exponent $\delta^\ast>0$, we compute
\begin{align}\label{6}
\sum_{j=1}^\infty\big(j+1\big)\cdot\frac{w(Q(y,\ell))^{1/{\alpha}-1/q}}{w(Q(y,2^{j+1}\sqrt{n}\ell))^{1/{\alpha}-1/q}}
&\leq C\sum_{j=1}^\infty(j+1)\cdot\left(\frac{|Q(y,\ell)|}{|Q(y,2^{j+1}\sqrt{n}\ell)|}\right)^{\delta^\ast(1/{\alpha}-1/q)}\notag\\
&=C\sum_{j=1}^\infty(j+1)\cdot\left(\frac{1}{2^{(j+1)n}}\right)^{\delta^\ast(1/{\alpha}-1/q)}\notag\\
&\leq C.
\end{align}
Notice that the exponent $\delta^\ast(1/{\alpha}-1/q)$ is positive because $\alpha<q$, which guarantees that the last series is convergent. Thus, by taking the $L^q(\mu)$-norm of both sides of \eqref{L}(with respect to the variable $y$), and then using Minkowski's inequality, \eqref{L1} and \eqref{L2}, we obtain
\begin{equation*}
\begin{split}
&\Big\|w(Q(y,\ell))^{1/{\alpha}-1/p-1/q}\big\|[b,\mathcal T](f)\cdot\chi_{Q(y,\ell)}\big\|_{WL^p(w)}\Big\|_{L^q(\mu)}\\
&\leq\big\|L_1(y,\ell)\big\|_{L^q(\mu)}+\big\|L_2(y,\ell)\big\|_{L^q(\mu)}\\
&\leq C\Big\|w(Q(y,2\sqrt{n}\ell))^{1/{\alpha}-1/p-1/q}\big\|f\cdot\chi_{Q(y,2\sqrt{n}\ell)}\big\|_{L^p(\nu)}\Big\|_{L^q(\mu)}\\
\end{split}
\end{equation*}
\begin{equation*}
\begin{split}
&+C\sum_{j=1}^\infty\Big\|w\big(Q(y,2^{j+1}\sqrt{n}\ell)\big)^{1/{\alpha}-1/p-1/q}\big\|f\cdot\chi_{Q(y,2^{j+1}\sqrt{n}\ell)}\big\|_{L^p(\nu)}\Big\|_{L^q(\mu)}\\
&\times\big(j+1\big)\cdot\frac{w(Q(y,\ell))^{1/{\alpha}-1/q}}{w(Q(y,2^{j+1}\sqrt{n}\ell))^{1/{\alpha}-1/q}}\\
&\leq C\big\|f\big\|_{(L^p,L^q)^{\alpha}(\nu,w;\mu)}+C\big\|f\big\|_{(L^p,L^q)^{\alpha}(\nu,w;\mu)}
\times\sum_{j=1}^\infty\big(j+1\big)\cdot\frac{w(Q(y,\ell))^{1/{\alpha}-1/q}}{w(Q(y,2^{j+1}\sqrt{n}\ell))^{1/{\alpha}-1/q}}\\
&\leq C\big\|f\big\|_{(L^p,L^q)^{\alpha}(\nu,w;\mu)},
\end{split}
\end{equation*}
where the last inequality is due to \eqref{6}. We conclude the proof of Theorem \ref{mainthm:7} by taking the supremum over all $\ell>0$.
\end{proof}

\begin{proof}[Proof of Theorem $\ref{mainthm:8}$]
Let $1<p\leq\alpha<q\leq\infty$ and $f\in(L^p,L^q)^{\alpha}(\nu,w;\mu)$ with $w\in A_\infty$ and $\mu\in\Delta_2$. For any fixed cube $Q=Q(y,\ell)$ in $\mathbb R^n$, as usual, we decompose $f$ as
\begin{equation*}
f=f_1+f_2,\qquad f_1=f\cdot\chi_{2Q},\quad  f_2=f\cdot\chi_{(2Q)^c},
\end{equation*}
where $2Q=Q(y,2\sqrt{n}\ell)$. Then for given $y\in\mathbb R^n$ and $\ell>0$, we write
\begin{align}\label{Lprime}
&w(Q(y,\ell))^{1/{\alpha}-1/p-1/q}\big\|[b,\mathcal T_{\gamma}](f)\cdot\chi_{Q(y,\ell)}\big\|_{WL^p(w)}\notag\\
&\leq 2\cdot w(Q(y,\ell))^{1/{\alpha}-1/p-1/q}\big\|[b,\mathcal T_{\gamma}](f_1)\cdot\chi_{Q(y,\ell)}\big\|_{WL^p(w)}\notag\\
&+2\cdot w(Q(y,\ell))^{1/{\alpha}-1/p-1/q}\big\|[b,\mathcal T_{\gamma}](f_2)\cdot\chi_{Q(y,\ell)}\big\|_{WL^p(w)}\notag\\
&:=L'_1(y,\ell)+L'_2(y,\ell).
\end{align}
Since $w\in A_\infty$, we know that $w\in\Delta_2$. By the assumption \eqref{assump4.2} and inequality (\ref{doubling3}), we get
\begin{align}\label{L1prime}
L'_1(y,\ell)&\leq 2\cdot w(Q(y,\ell))^{1/{\alpha}-1/p-1/q}\big\|[b,\mathcal T_{\gamma}](f_1)\big\|_{WL^p(w)}\notag\\
&\leq C\cdot w(Q(y,\ell))^{1/{\alpha}-1/p-1/q}
\bigg(\int_{Q(y,2\sqrt{n}\ell)}|f(x)|^p\nu(x)\,dx\bigg)^{1/p}\notag\\
&=C\cdot w(Q(y,2\sqrt{n}\ell))^{1/{\alpha}-1/p-1/q}\big\|f\cdot\chi_{Q(y,2\sqrt{n}\ell)}\big\|_{L^p(\nu)}\notag\\
&\times \frac{w(Q(y,\ell))^{1/{\alpha}-1/p-1/q}}{w(Q(y,2\sqrt{n}\ell))^{1/{\alpha}-1/p-1/q}}\notag\\
&\leq C\cdot w(Q(y,2\sqrt{n}\ell))^{1/{\alpha}-1/p-1/q}\big\|f\cdot\chi_{Q(y,2\sqrt{n}\ell)}\big\|_{L^p(\nu)}.
\end{align}
Next we estimate the other term $L'_2(y,\ell)$. For any $x\in Q(y,\ell)$, from the condition \eqref{frac sublinear commutator}, one can see that
\begin{equation*}
\begin{split}
\big|[b,\mathcal T_{\gamma}](f_2)(x)\big|
&\leq \big|b(x)-b_{Q(y,\ell)}\big|\cdot\big|\mathcal T_\gamma(f_2)(x)\big|
+\Big|\mathcal T_\gamma\big([b_{Q(y,\ell)}-b]f_2\big)(x)\Big|\\
&:=\widetilde{\xi}_\ast(x)+\widetilde{\eta}_\ast(x).
\end{split}
\end{equation*}
Consequently, we can further divide $L'_2(y,\ell)$ into two parts:
\begin{equation*}
\begin{split}
L'_2(y,\ell)\leq&4\cdot w(Q(y,\ell))^{1/{\alpha}-1/p-1/q}\big\|\widetilde{\xi}_\ast(\cdot)\cdot\chi_{Q(y,\ell)}\big\|_{WL^p(w)}\\
&+4\cdot w(Q(y,\ell))^{1/{\alpha}-1/p-1/q}\big\|\widetilde{\eta}_\ast(\cdot)\cdot\chi_{Q(y,\ell)}\big\|_{WL^p(w)}\\
:=&L'_3(y,\ell)+L'_4(y,\ell).
\end{split}
\end{equation*}
For the term $L'_3(y,\ell)$, it follows directly from Chebyshev's inequality and estimate \eqref{alpha1} that
\begin{equation*}
\begin{split}
L'_3(y,\ell)&\leq4\cdot w(Q(y,\ell))^{1/{\alpha}-1/p-1/q}\bigg(\int_{Q(y,\ell)}\big|\widetilde{\xi}_\ast(x)\big|^pw(x)\,dx\bigg)^{1/p}\\
&\leq C\cdot w(Q(y,\ell))^{1/{\alpha}-1/p-1/q}\bigg(\int_{Q(y,\ell)}\big|b(x)-b_{Q(y,\ell)}\big|^pw(x)\,dx\bigg)^{1/p}\\
&\times\sum_{j=1}^\infty\frac{1}{|Q(y,2^{j+1}\sqrt{n}\ell)|^{1-\gamma/n}}\int_{Q(y,2^{j+1}\sqrt{n}\ell)}|f(z)|\,dz\\
&\leq C\cdot w(Q(y,\ell))^{1/{\alpha}-1/q}
\sum_{j=1}^\infty\frac{1}{|Q(y,2^{j+1}\sqrt{n}\ell)|^{1-\gamma/n}}\int_{Q(y,2^{j+1}\sqrt{n}\ell)}|f(z)|\,dz,
\end{split}
\end{equation*}
where in the last inequality we have used the fact that $w\in A_{\infty}$ and Lemma \ref{BMO}$(ii)$. Arguing as in the proof of Theorem \ref{mainthm:6}, we can also obtain that
\begin{equation*}
L'_3(y,\ell)\leq C\sum_{j=1}^\infty w\big(Q(y,2^{j+1}\sqrt{n}\ell)\big)^{1/{\alpha}-1/p-1/q}\big\|f\cdot\chi_{Q(y,2^{j+1}\sqrt{n}\ell)}\big\|_{L^p(\nu)}
\cdot\frac{w(Q(y,\ell))^{1/{\alpha}-1/q}}{w(Q(y,2^{j+1}\sqrt{n}\ell))^{1/{\alpha}-1/q}}.
\end{equation*}
Let us now estimate the term $L'_4(y,\ell)$. As it was proved in Theorem \ref{mainthm:4}, the following pointwise estimate holds by the size condition \eqref{frac sublinear}.
\begin{equation*}
\begin{split}
\widetilde{\eta}_\ast(x)&=\Big|\mathcal T_\gamma\big([b_{Q(y,\ell)}-b]f_2\big)(x)\Big|\\
&\leq C\sum_{j=1}^\infty\frac{1}{|Q(y,2^{j+1}\sqrt{n}\ell)|^{1-\gamma/n}}\int_{Q(y,2^{j+1}\sqrt{n}\ell)}\big|b(z)-b_{Q(y,\ell)}\big|\cdot|f(z)|\,dz.
\end{split}
\end{equation*}
This, together with Chebyshev's inequality implies that
\begin{equation*}
\begin{split}
L'_4(y,\ell)&\leq4\cdot w(Q(y,\ell))^{1/{\alpha}-1/p-1/q}\bigg(\int_{Q(y,\ell)}\big|\widetilde{\eta}_\ast(x)\big|^pw(x)\,dx\bigg)^{1/p}\\
&\leq C\cdot w(Q(y,\ell))^{1/{\alpha}-1/q}
\sum_{j=1}^\infty\frac{1}{|Q(y,2^{j+1}\sqrt{n}\ell)|^{1-\gamma/n}}\int_{Q(y,2^{j+1}\sqrt{n}\ell)}\big|b(z)-b_{Q(y,\ell)}\big|\cdot|f(z)|\,dz\\
&\leq C\cdot w(Q(y,\ell))^{1/{\alpha}-1/q}
\sum_{j=1}^\infty\frac{1}{|Q(y,2^{j+1}\sqrt{n}\ell)|^{1-\gamma/n}}\int_{Q(y,2^{j+1}\sqrt{n}\ell)}\big|b(z)-b_{Q(y,2^{j+1}\sqrt{n}\ell)}\big|\cdot|f(z)|\,dz\\
&+C\cdot w(Q(y,\ell))^{1/{\alpha}-1/q}
\sum_{j=1}^\infty\frac{1}{|Q(y,2^{j+1}\sqrt{n}\ell)|^{1-\gamma/n}}\int_{Q(y,2^{j+1}\sqrt{n}\ell)}\big|b_{Q(y,2^{j+1}\sqrt{n}\ell)}-b_{Q(y,\ell)}\big|\cdot|f(z)|\,dz\\
&:=L'_5(y,\ell)+L'_6(y,\ell).
\end{split}
\end{equation*}
An application of H\"older's inequality leads to that
\begin{equation*}
\begin{split}
L'_5(y,\ell)&\leq C\cdot w(Q(y,\ell))^{1/{\alpha}-1/q}
\sum_{j=1}^\infty\frac{1}{|Q(y,2^{j+1}\sqrt{n}\ell)|^{1-\gamma/n}}\bigg(\int_{Q(y,2^{j+1}\sqrt{n}\ell)}|f(z)|^p\nu(z)\,dz\bigg)^{1/p}\\
&\times\bigg(\int_{Q(y,2^{j+1}\sqrt{n}\ell)}\big|b(z)-b_{Q(y,2^{j+1}\sqrt{n}\ell)}\big|^{p'}\nu(z)^{-p'/p}\,dz\bigg)^{1/{p'}}\\
&=C\cdot w(Q(y,\ell))^{1/{\alpha}-1/q}\sum_{j=1}^\infty\frac{\big\|f\cdot\chi_{Q(y,2^{j+1}\sqrt{n}\ell)}\big\|_{L^p(\nu)}}{|Q(y,2^{j+1}\sqrt{n}\ell)|^{1-\gamma/n}}\\
&\times\big|Q(y,2^{j+1}\sqrt{n}\ell)\big|^{1/{p'}}\Big\|\big[b-b_{Q(y,2^{j+1}\sqrt{n}\ell)}\big]\cdot\nu^{-1/p}\Big\|_{\mathcal C,Q(y,2^{j+1}\sqrt{n}\ell)},
\end{split}
\end{equation*}
where $\mathcal C(t)=t^{p'}$ is a Young function. Moreover, in view of \eqref{U2} and \eqref{final}, we can deduce that
\begin{equation*}
\begin{split}
L'_5(y,\ell)&\leq C\|b\|_*\cdot w(Q(y,\ell))^{1/{\alpha}-1/q}
\sum_{j=1}^\infty\frac{\big\|f\cdot\chi_{Q(y,2^{j+1}\sqrt{n}\ell)}\big\|_{L^p(\nu)}}{|Q(y,2^{j+1}\sqrt{n}\ell)|^{1/p-\gamma/n}}
\cdot\Big\|\nu^{-1/p}\Big\|_{\mathcal A,Q(y,2^{j+1}\sqrt{n}\ell)}\\
&=C\|b\|_*\sum_{j=1}^\infty w\big(Q(y,2^{j+1}\sqrt{n}\ell)\big)^{1/{\alpha}-1/p-1/q}\big\|f\cdot\chi_{Q(y,2^{j+1}\sqrt{n}\ell)}\big\|_{L^p(\nu)}
\cdot\frac{w(Q(y,\ell))^{1/{\alpha}-1/q}}{w(Q(y,2^{j+1}\sqrt{n}\ell))^{1/{\alpha}-1/q}}\\
&\times\frac{w(Q(y,2^{j+1}\sqrt{n}\ell))^{1/p}}{|Q(y,2^{j+1}\sqrt{n}\ell)|^{1/p-\gamma/n}}
\cdot\Big\|\nu^{-1/p}\Big\|_{\mathcal A,Q(y,2^{j+1}\sqrt{n}\ell)}\\
&\leq C\|b\|_*\sum_{j=1}^\infty w\big(Q(y,2^{j+1}\sqrt{n}\ell)\big)^{1/{\alpha}-1/p-1/q}\big\|f\cdot\chi_{Q(y,2^{j+1}\sqrt{n}\ell)}\big\|_{L^p(\nu)}
\cdot\frac{w(Q(y,\ell))^{1/{\alpha}-1/q}}{w(Q(y,2^{j+1}\sqrt{n}\ell))^{1/{\alpha}-1/q}}\\
&\times\frac{|Q(y,2^{j+1}\sqrt{n}\ell)|^{1/{(r'p)}}}{|Q(y,2^{j+1}\sqrt{n}\ell)|^{1/p-\gamma/n}}
\bigg(\int_{Q(y,2^{j+1}\sqrt{n}\ell)}w(z)^r\,dz\bigg)^{1/{(rp)}}\cdot\Big\|\nu^{-1/p}\Big\|_{\mathcal A,Q(y,2^{j+1}\sqrt{n}\ell)}\\
&\leq C\|b\|_*\sum_{j=1}^\infty w\big(Q(y,2^{j+1}\sqrt{n}\ell)\big)^{1/{\alpha}-1/p-1/q}\big\|f\cdot\chi_{Q(y,2^{j+1}\sqrt{n}\ell)}\big\|_{L^p(\nu)}
\cdot\frac{w(Q(y,\ell))^{1/{\alpha}-1/q}}{w(Q(y,2^{j+1}\sqrt{n}\ell))^{1/{\alpha}-1/q}}.
\end{split}
\end{equation*}
The last inequality is obtained by the $A_p$-type condition $(\S\S')$ on $(w,\nu)$. Let us estimate the last term $L'_6(y,\ell)$. Applying Lemma \ref{BMO}$(i)$ and H\"older's inequality, we get
\begin{equation*}
\begin{split}
L'_6(y,\ell)&\leq C\cdot w(Q(y,\ell))^{1/{\alpha}-1/q}
\sum_{j=1}^\infty\frac{(j+1)\|b\|_*}{|Q(y,2^{j+1}\sqrt{n}\ell)|^{1-\gamma/n}}\int_{Q(y,2^{j+1}\sqrt{n}\ell)}|f(z)|\,dz\\
&\leq C\cdot w(Q(y,\ell))^{1/{\alpha}-1/q}
\sum_{j=1}^\infty\frac{(j+1)\|b\|_*}{|Q(y,2^{j+1}\sqrt{n}\ell)|^{1-\gamma/n}}\bigg(\int_{Q(y,2^{j+1}\sqrt{n}\ell)}|f(z)|^p\nu(z)\,dz\bigg)^{1/p}\\
&\times\bigg(\int_{Q(y,2^{j+1}\sqrt{n}\ell)}\nu(z)^{-p'/p}\,dz\bigg)^{1/{p'}}\\
\end{split}
\end{equation*}
\begin{equation*}
\begin{split}
&=C\|b\|_*\sum_{j=1}^\infty w\big(Q(y,2^{j+1}\sqrt{n}\ell)\big)^{1/{\alpha}-1/p-1/q}\big\|f\cdot\chi_{Q(y,2^{j+1}\sqrt{n}\ell)}\big\|_{L^p(\nu)}\\
&\times\big(j+1\big)\cdot\frac{w(Q(y,\ell))^{1/{\alpha}-1/q}}{w(Q(y,2^{j+1}\sqrt{n}\ell))^{1/{\alpha}-1/q}}
\cdot\frac{w(Q(y,2^{j+1}\sqrt{n}\ell))^{1/p}}{|Q(y,2^{j+1}\sqrt{n}\ell)|^{1-\gamma/n}}\bigg(\int_{Q(y,2^{j+1}\sqrt{n}\ell)}\nu(z)^{-p'/p}\,dz\bigg)^{1/{p'}}.
\end{split}
\end{equation*}
Also observe that the condition $(\S\S')$ is stronger than the condition $(\S')$. Using this fact along with \eqref{U2}, we get
\begin{equation*}
\begin{split}
L'_6(y,\ell)&\leq C\|b\|_*\sum_{j=1}^\infty w\big(Q(y,2^{j+1}\sqrt{n}\ell)\big)^{1/{\alpha}-1/p-1/q}\big\|f\cdot\chi_{Q(y,2^{j+1}\sqrt{n}\ell)}\big\|_{L^p(\nu)}\\
&\times\big(j+1\big)\cdot\frac{w(Q(y,\ell))^{1/{\alpha}-1/q}}{w(Q(y,2^{j+1}\sqrt{n}\ell))^{1/{\alpha}-1/q}}\\
&\times\frac{|Q(y,2^{j+1}\sqrt{n}\ell)|^{1/{(r'p)}}}{|Q(y,2^{j+1}\sqrt{n}\ell)|^{1-\gamma/n}}
\bigg(\int_{Q(y,2^{j+1}\sqrt{n}\ell)}w(z)^r\,dz\bigg)^{1/{(rp)}}
\bigg(\int_{Q(y,2^{j+1}\sqrt{n}\ell)}\nu(z)^{-p'/p}\,dz\bigg)^{1/{p'}}\\
&\leq C\|b\|_*\sum_{j=1}^\infty w\big(Q(y,2^{j+1}\sqrt{n}\ell)\big)^{1/{\alpha}-1/p-1/q}\big\|f\cdot\chi_{Q(y,2^{j+1}\sqrt{n}\ell)}\big\|_{L^p(\nu)}\\
&\times\big(j+1\big)\cdot\frac{w(Q(y,\ell))^{1/{\alpha}-1/q}}{w(Q(y,2^{j+1}\sqrt{n}\ell))^{1/{\alpha}-1/q}}.
\end{split}
\end{equation*}
Summing up all the above estimates, we conclude that
\begin{align}\label{L2prime}
L'_2(y,\ell)&\leq C\sum_{j=1}^\infty w\big(Q(y,2^{j+1}\sqrt{n}\ell)\big)^{1/{\alpha}-1/p-1/q}\big\|f\cdot\chi_{Q(y,2^{j+1}\sqrt{n}\ell)}\big\|_{L^p(\nu)}\notag\\
&\times\big(j+1\big)\cdot\frac{w(Q(y,\ell))^{1/{\alpha}-1/q}}{w(Q(y,2^{j+1}\sqrt{n}\ell))^{1/{\alpha}-1/q}}.
\end{align}
Thus, by taking the $L^q(\mu)$-norm of both sides of \eqref{Lprime}(with respect to the variable $y$), and then using Minkowski's inequality, \eqref{L1prime} and \eqref{L2prime}, we finally obtain
\begin{equation*}
\begin{split}
&\Big\|w(Q(y,\ell))^{1/{\alpha}-1/p-1/q}\big\|[b,\mathcal T_{\gamma}](f)\cdot\chi_{Q(y,\ell)}\big\|_{WL^p(w)}\Big\|_{L^q(\mu)}\\
&\leq\big\|L'_1(y,\ell)\big\|_{L^q(\mu)}+\big\|L'_2(y,\ell)\big\|_{L^q(\mu)}\\
&\leq C\Big\|w(Q(y,2\sqrt{n}\ell))^{1/{\alpha}-1/p-1/q}\big\|f\cdot\chi_{Q(y,2\sqrt{n}\ell)}\big\|_{L^p(\nu)}\Big\|_{L^q(\mu)}\\
&+C\sum_{j=1}^\infty\Big\|w\big(Q(y,2^{j+1}\sqrt{n}\ell)\big)^{1/{\alpha}-1/p-1/q}\big\|f\cdot\chi_{Q(y,2^{j+1}\sqrt{n}\ell)}\big\|_{L^p(\nu)}\Big\|_{L^q(\mu)}\\
&\times\big(j+1\big)\cdot\frac{w(Q(y,\ell))^{1/{\alpha}-1/q}}{w(Q(y,2^{j+1}\sqrt{n}\ell))^{1/{\alpha}-1/q}}\\
&\leq C\big\|f\big\|_{(L^p,L^q)^{\alpha}(\nu,w;\mu)}+C\big\|f\big\|_{(L^p,L^q)^{\alpha}(\nu,w;\mu)}
\times\sum_{j=1}^\infty\big(j+1\big)\cdot\frac{w(Q(y,\ell))^{1/{\alpha}-1/q}}{w(Q(y,2^{j+1}\sqrt{n}\ell))^{1/{\alpha}-1/q}}\\
&\leq C\big\|f\big\|_{(L^p,L^q)^{\alpha}(\nu,w;\mu)},
\end{split}
\end{equation*}
where the last inequality is due to \eqref{6}. We therefore conclude the proof of Theorem \ref{mainthm:8} by taking the supremum over all $\ell>0$.
\end{proof}

\end{document}